\documentclass[11pt,a4paper,reqno]{amsart}

\usepackage{epsfig}
\usepackage{mathrsfs} 
\usepackage{amsfonts}
\usepackage{amsmath}
\usepackage{amssymb}
\usepackage{amsthm}
\usepackage{textcomp}
\usepackage{times}
\usepackage{color}
\usepackage{enumerate}
\usepackage[T1]{fontenc}
\usepackage[latin1]{inputenc}
\usepackage{graphicx}
\usepackage{nicefrac}
\usepackage[includehead,includefoot,margin=25mm]{geometry}
\usepackage[all]{xy}
\usepackage[colorlinks=true,linkcolor=blue,urlcolor=blue, citecolor=blue]{hyperref}

\makeatletter
\numberwithin{equation}{section}
\numberwithin{figure}{section}
\theoremstyle{plain}
\newtheorem{thm}{\protect\theoremname}[section]
\theoremstyle{remark}
\newtheorem{rem}[thm]{\protect\remarkname}
\theoremstyle{definition}
\newtheorem{defn}[thm]{\protect\definitionname}
\theoremstyle{plain}
\newtheorem{lem}[thm]{\protect\lemmaname}
\theoremstyle{plain}
\newtheorem{prop}[thm]{\protect\propositionname}
\newtheorem{cor}[thm]{\protect\corollaryname}
\theoremstyle{remark}

\theoremstyle{definition}
\newtheorem{example}[thm]{\protect\examplename}

\theoremstyle{plain}

\makeatother

\usepackage[english]{babel}
\providecommand{\definitionname}{Definition}
\providecommand{\corollaryname}{Corollary}
\providecommand{\examplename}{Example}
\providecommand{\lemmaname}{Lemma}
\providecommand{\notationname}{Notation}
\providecommand{\propositionname}{Proposition}
\providecommand{\remarkname}{Remark}
\providecommand{\theoremname}{Theorem}

\makeatletter \newcommand*\bigcdot{\mathpalette\bigcdot@{.5}} \newcommand*\bigcdot@[2]{\mathbin{\vcenter{\hbox{\scalebox{#2}{$\m@th#1\bullet$}}}}} \makeatother

\definecolor{liendo}{rgb}{0.0, 0.5, 0.0}

\begin{document}

\title[Topologically integrable derivations and additive group actions on affine ind-schemes]{Topologically integrable derivations and \\ additive group actions on affine ind-schemes}

\author{Roberto D\'iaz}
\address{Departamento de Matem\'aticas, Facultad de Ciencias, Universidad de La Serena, Juan
Cisternas 1200, La Serena, Chile.}%
\email{roberto.diazv1@userena.cl}

\author{Adrien Dubouloz}   \address{Laboratoire de Math\'ematique et Applications, UMR 7348 CNRS, Universit\'e de Poitiers, 86000 Poitiers \\ ~\indent Universit\'e Bourgogne Europe, CNRS, IMB UMR 5584, 21000 Dijon, France}
\email{adrien.dubouloz@math.cnrs.fr}

\author{Alvaro Liendo} %
\address{Instituto de Matem\'atica y F\'\i sica, Universidad de Talca,
  Casilla 721, Talca, Chile.}%
\email{aliendo@inst-mat.utalca.cl}

\thanks{{\it 2020 Mathematics Subject
 Classification}: 13J10, 13N15, 14R20, 14L30.\\
 \mbox{\hspace{11pt}}{\it Key words}: affine ind-schemes, additive group actions, complete topological rings, restricted power series, restricted exponential homomorphisms, topologically integrable derivations.\\
 \mbox{\hspace{11pt}} The first author was supported by ANID via Proyecto Fondecyt Postdoctorado N\textsuperscript{o}3230406. The second author was partially supported by ANID via Proyecto Fondecyt Regular N\textsuperscript{o} 1240101 and Fondecyt Exploraci\'on 13250049.}

\begin{abstract} 
We develop a theory of additive group actions on affine ind-schemes through a purely algebraic and topological framework. Affine ind-schemes are described via complete, second-countable, linearly topologized rings, and actions of the additive group are encoded by restricted exponential homomorphisms. We introduce the notion of a topologically integrable derivation, a continuous derivation whose formal exponential converges in the sense of restricted power series, and show that this notion provides the correct extension of locally nilpotent derivations to the infinite-dimensional setting. Our first main result establishes a one-to-one correspondence between topologically integrable derivations and additive group actions on affine ind-schemes, extending the classical correspondence for affine varieties. We then investigate the structure of such actions admitting a slice. In this context, we prove an ind-scheme analog of the classical slice theorem: if an additive group action admits a slice, then the underlying affine ind-scheme is equivariantly isomorphic to a product with the affine line, and the action is given by translation on the second factor. Several examples illustrate the necessity of the topological hypotheses and highlight phenomena absent in the finite-type case.
\end{abstract}

\maketitle

\section*{Introduction}

For the sake of clarity, the discussion in this introduction is confined to schemes over a field $k$ of characteristic zero, although all results in the paper are established in a more general setting.

Actions of the additive group $\mathbb{G}_{a,k}$ on affine algebraic varieties have a deep classical connection with derivations of their coordinate rings. In particular, a regular $\mathbb{G}_{a,k}$-action on an affine $k$-variety $\operatorname{Spec}(R)$ is equivalent to the existence of a nonzero locally nilpotent $k$-derivation on $R$, reflecting the fact that a one-parameter unipotent group is generated infinitesimally by a nilpotent vector field. This correspondence, which originates from the work of Rentschler, Miyanishi, and others, has become a cornerstone of affine algebraic geometry (see, e.g., Freudenburg's exposition in \cite{F06}). It not only provides an algebraic way to handle  $\mathbb{G}_{a,k}$-actions via locally nilpotent derivations (LNDs), but also leads to structure theorems: for instance, if a $\mathbb{G}_{a,k}$-action on a variety admits a slice (an invariant affine-linear coordinate), then the variety is isomorphic to a trivial $\mathbb{G}_{a,k}$-bundle over its quotient. In coordinates, the presence of a slice $x$ means that $R \simeq R^{\mathbb{G}_{a,k}}[s]$ as rings (with $\mathbb{G}_{a,k}$ acting by translation in the $s$-direction).

In this paper, we extend this classical theory to the setting of affine ind-schemes. An affine ind-scheme in our sense is a filtered direct limit of affine schemes along closed embeddings, and we restrict to the case where a countable ascending chain of affine schemes realizes the limit. Equivalently, the coordinate ring $B$ of such an ind-scheme can be written as an inverse limit of a countable sequence of finitely generated algebras, endowed with the inverse limit topology making $B$ into a complete, Hausdorff, linearly topologized ring. We model a $\mathbb{G}_{a,k}$-action on $\operatorname{Spec}(B)$ by a restricted exponential homomorphism $\mathrm{e}\colon B \to B\{T\}$, where $B\{T\}\subset B[[T]]$ is the ring of restricted power series in an indeterminate $T$ with coefficients in $B$ (informally, formal power series that converge under the topology of $B$, see Definition~\ref{def:Restricted-Exp-Map} for details). This algebra homomorphism $\mathrm{e}$ plays the role of the coaction map for $\mathbb{G}_{a,k}$. Indeed, it satisfies the usual group law identities by means of commutative diagrams (cf. Definition~\ref{def:Restricted-Exp-Map}) and can be viewed as the exponential of a derivation. 

Our first main result establishes that the classical correspondence between $\mathbb{G}_{a,k}$-actions and locally nilpotent derivations persists in this infinite-dimensional context, once derivations are interpreted in a suitable topological sense. Indeed, we introduce the notion of a topologically integrable derivation on a complete topological algebra $B$: this is a continuous derivation $\partial: B \to B$ whose formal exponential 
$$\exp(T\partial)\colon B\to B[[T]] \quad\mbox{given by}\quad f \mapsto \sum_{i\ge0} \frac{1}{i!}T^i\partial^i(f)$$ 
converges in $B\{T\}$ to a restricted exponential homomorphism (see Definition~\ref{def:Pro-LND} for the precise details). Theorem~\ref{top-int-1-1-exp} provides a one-to-one correspondence between topologically integrable derivations on $B$ and restricted exponential homomorphisms $\mathrm{e}: B \to B\{T\}$. This extends the classical correspondence between algebraic $\mathbb{G}_{a,k}$-actions and LNDs to affine ind-schemes. 

We stress that the additional topological hypotheses in Definition~\ref{def:Pro-LND} (notably the uniform equicontinuity of ${D(i)}$ and its pointwise convergence to zero) are essential in the infinite-dimensional setting: unlike the finite-type case, a derivation can be pointwise nilpotent on $B$ without being globally integrable unless a uniform convergence criterion is met. For example, we exhibit a continuous derivation on a certain countably-generated power series algebra that kills each element of $B$ to order $n$ for some $n$ (depending on the element) while not admitting a continuous restricted exponential series within $B$, see Example~\ref{ex:non-contconv-der}. Such pathologies do not occur in the classical (discrete) theory, and our framework of topologically integrable derivations is designed to exclude them.

Having established the fundamental correspondence, we then investigate $\mathbb{G}_{a,k}$-actions on affine ind-schemes from a structural perspective. A central concept is that of a slice for the action, generalizing the classical notion mentioned above. In our topological context, a local slice is an element $s\in B$ whose image $\mathrm{e}(s)\in B\{T\}$ is a non-constant polynomial of minimal possible degree, and a slice is a local slice $s$ such that $\mathrm{e}(s)$ is furthermore a monic polynomial. Intuitively, $s$ provides a ``transverse'' coordinate to the $\mathbb{G}_{a,k}$-orbits, being as simple as possible under the given action. Our second main result, Theorem \ref{Theorem: Slice}, is a structure theorem for $\mathbb{G}_{a,k}$-actions admitting a slice. It states that if $\mathrm{e}: B \to B\{T\}$ is a restricted exponential homomorphism (i.e., an additive group action) and $s\in B$ is a slice, then the slice map $\theta: B^\mathrm{e}[X] \to B$ sending $X$ to $s$ is an isomorphism of topological $ B^\mathrm{e}$-algebras. Here $ B^\mathrm{e} = \{b\in B: \mathrm{e}(b) = b\}$ is the ring of invariants of the action. In particular, $\operatorname{Spec}(B)$ is isomorphic to $\operatorname{Spec}( B^\mathrm{e}) \times \mathbb{A}^1$ (as ind-schemes) in such a way that the given $\mathbb{G}_{a,k}$-action corresponds to the natural translation action on the $\mathbb{A}^1$-factor. This result can be viewed as an ind-scheme analog of the classical slice theorem for additive group actions, recovering the familiar fact $R \cong R^{\mathbb{G}_{a,k}}[X]$ in the discrete case as a special scenario. 

\medskip 

The paper is organized as follows. In Section \ref{sec:1}, we recall the description of restricted strict affine ind-schemes in terms of complete linearly topologized rings, which provides the algebraic framework used throughout the article. Section \ref{sec:2} introduces restricted exponential homomorphisms and topologically integrable iterated higher derivations, and establishes the correspondence between these notions, leading to Theorem \ref{top-int-1-1-exp}. In Section \ref{sec:3}, we develop basic structural properties of restricted exponential homomorphisms, including invariants, localization, passage to quotients, and compatibility with limits and automorphisms. Section \ref{sec:4} is devoted to restricted exponential homomorphisms admitting local slices and contains the proof of the main structure theorem, Theorem \ref{Theorem: Slice}. Section \ref{sec:5} discusses further examples and applications illustrating the theory. The appendices collect background material and technical results on topological rings, restricted power series, and related completeness properties used throughout the paper.




\section{Preliminaries}\label{sec:1}

\subsection{Restricted strict affine ind-schemes as complete second-countable linearly topologized rings}\label{sub:abstract-nonsense}

The purpose of this section is to provide the reader with a summary that contextualizes our article within the existing panorama of different notions of ind-varieties and ind-schemes that have already been considered within the scope of affine geometry, and to motivate the choice to formulate the results of the article in the formalism of topological rings. We do not claim any originality in the following presentation; all the elements presented are taken from \cite[Expos\'e 1]{SGA4-I},  \cite[Appendix]{AB-69}, \cite[Section A]{G-60} and  \cite[Section 6]{KS06} for ind-objects and pro-objects in a category, and from \cite[Chapter 0, $\S$ 7]{GD} for the principle of the correspondence between certain classes of pro-rings and separated and complete topological rings. We refer the reader to these references for all the detailed definitions and precise  statements. 

\medskip 

There is a standard notion of an ind-object in any category $\mathcal{C}$: the category $\mathrm{Ind}(\mathcal{C})$ of ind-objects of $\mathcal{C}$ is defined as the full subcategory of the category $\mathrm{Fun}(\mathcal{C}^{\mathrm{opp}},(\mathrm{Sets}))$ of set-valued contravariant functors on $\mathcal{C}$ whose objects $F$ are small filtered colimits of  representable functors, i.e $F\cong \mathrm{colim}_{d\in D}(\upsilon (\alpha(d)))$ where $D$ is a small filtered category, $\alpha:D\to \mathcal{C}$ is a functor and $\upsilon:\mathcal{C}\to \mathrm{Fun}(\mathcal{C}^{\mathrm{opp}},(\mathrm{Sets}))$, $X\mapsto \mathrm{Hom}(-,X)$ is the Yoneda embedding. Taking for $\mathcal{C}$ the category $(\mathrm{Aff})$ of affine schemes  gives in particular a notion of affine ind-scheme. But the notion we consider in this article is more restrictive: these are those
affine ind-schemes $F$ for which the index category $D$ can be taken equal to that of representable closed sub-functors\footnote{That is, a representable sub-functor $\upsilon(X)\to F$ such that for every morphism of functors $\upsilon(S)\to F$, where $S$ is an affine scheme, the functor $\upsilon(S)\times_F \upsilon(X)$ is representable, say isomorphic to $\upsilon(T)$ for some affine scheme $T$, and the morphism $\upsilon(S)\times_F \upsilon(X)\to \upsilon(X)$ is equal to $\upsilon(j)$ for a closed immersion $j:T\to X$.} of $F$ and satisfy the additional condition to admit a directed countable cofinal full subcategory $I$. Since every such subcategory $I$ contains a cofinal directed full subcategory isomorphic to a chain $J=\{0 < 1 < \cdots < n < \cdots \}$, finite or not, of consecutive integers, the combination of the two conditions is equivalent to requiring that $F\cong \varinjlim_{j\in J} \upsilon(X_j)$ for a chain of closed immersions $(X_0\hookrightarrow X_1\hookrightarrow \cdots \hookrightarrow X_j \hookrightarrow \cdots )_{j\in J}$ between affine schemes. To avoid confusion with other concepts of ind-varieties and ind-schemes in the context of affine algebraic geometry mentioned in the introduction that are circulating in the literature, it might be convenient to call this full subcategory of $\mathrm{Ind(Aff)}$ the category of \emph{restricted strict affine ind-schemes}.

The anti-equivalence between the category $(\mathrm{Aff})$ and the
category $(\mathrm{Rings})$ of commutative unital rings extends to an anti-equivalence between $\mathrm{Ind(Aff)}$ and the category $\mathrm{Pro(Rings)}$ of pro-objects of the category 
$(\mathrm{Rings})$, defined in a dual way as the full sub-category of the category $\mathrm{Fun}((\mathrm{Rings}),(\mathrm{Sets}))^{\mathrm{opp}}$ 
opposite to that of covariant set-valued functors on $(\mathrm{Rings})$ whose objects  $G$ are small co-filtered limits of representable functors, i.e. $G\cong\lim_{d'\in D'}\upsilon'(\alpha'(d'))$, where $D'$ is a small co-filtered category, $\alpha':D'\to (\mathrm{Rings})$ is a functor and 
$\upsilon':(\mathrm{Rings})\to \mathrm{Fun}((\mathrm{Rings}),(\mathrm{Sets}))^{\mathrm{opp}}$, $R\mapsto \mathrm{Hom}(R,-)$ is the "opposite" Yoneda embedding.  Restricted strict affine ind-schemes correspond under this anti-equivalence to pro-rings $G$ for which the index category $D'$ can be taken equal to that of representable quotients of $G$ and satisfy the additional condition of existence of a chain $(\cdots \to R_j\to \cdots \to R_1\to R_0)_{j\in J}$ of surjective homomorphisms of rings such that $G\cong \varprojlim_{j\in J} \upsilon'(R_j)$. We call the corresponding full subcategory of $\mathrm{Pro(Rings)}$ the category of \emph{restricted strict pro-rings}.

The category $(\mathrm{Rings})$ being complete, the limit $\varprojlim_{d'\in D'} \alpha'(d')$
of any diagram $\alpha':D'\to (\mathrm{Rings})$ of surjective homomorphisms of rings
 indexed by a small co-filtered category $D'$  exists in $(\mathrm{Rings})$. But it is important to notice here that since
the Yoneda embedding $\upsilon'$ does not in general preserve limits, the limits $\varprojlim_{d'\in D'} \alpha'(d')$ in  $(\mathrm{Rings})$ and $\varprojlim_{d'\in D'} \upsilon'(\alpha'(d'))$ in $\mathrm{Pro(Rings)}$ are in general different, i.e., the canonical morphism of functors $$\upsilon'(\varprojlim_{d'\in D'} \alpha'(d'))=\mathrm{Hom}(\varprojlim_{d'\in D'} \alpha'(d'),-)\to  \varprojlim_{d'\in D'} \upsilon'(\alpha'(d'))=\varprojlim_{d'\in D'} \mathrm{Hom}(\alpha'(d'),-)$$ is in general not an isomorphism.

There is nevertheless a standard construction which allows to fully retain the richer information of the pro-ring $G=\varprojlim_{d'\in D'} \upsilon'(\alpha'(d'))$ in terms of a suitable additional structure of topological ring on the ring $\varprojlim_{d'\in D'} \alpha'(d')$.   Namely, considering each of the rings $\alpha'(d')$ as endowed with the discrete topology, one endows $\mathcal{R}_G:=\varprojlim_{d'\in D'} \alpha'(d')$ with the limit topology, that is, the coarsest topology making all the canonical surjections $p_{d'}:\mathcal{R}_G\to \alpha'(d')$, $d'\in D'$, continuous. The so-defined topology $\tau_G$ makes $\mathcal{R}_G$ into a separated linearly topologized commutative unital topological ring with a fundamental system $\mathrm{B}$ of open neighborhoods of $0_{\mathcal{R}_G}$ consisting of the ideal $\mathfrak{r}_{d'}=p_{d'}^{-1}(0_{\alpha'(d')})$, $d'\in D'$. Moreover, the underlying abelian topological group of $\mathcal{R}_G$, viewed as a uniform space with a base of uniformities indexed by the above open neighborhoods of $0_{\mathcal{R}_G}$, is complete. 
The condition of existence of a chain $(\cdots \to R_j\to \cdots\to R_1\to R_0)_{j\in J}$ of surjective homomorphisms of rings such that $G\cong \varprojlim_{j\in J} \upsilon'(R_j)$ is equivalent to requiring that $(\mathcal{R}_G,\tau_G)$ further admits a countable basis of open neighborhoods of $0_{\mathcal{R}_G}$, i.e. that $(\mathcal{R}_G,\tau_G)$ is a second-countable topological space . Conversely, every separated complete linearly topologized topological ring $(\mathcal{R},\tau)$ determines a pro-ring $$G_{(\mathcal{R},\tau)}:(\mathrm{Rings})\to (\mathrm{Sets}), A\mapsto \mathrm {CHom}(\mathcal{R},A)\cong\varprojlim_{\mathfrak{r}\in \mathrm{B}} \mathrm{Hom}(\mathcal{R}/\mathfrak{r},A),$$   
where $\mathrm{CHom}(\mathcal{R},A)$ is the set of continuous rings homomorphisms between $(\mathcal{R},\tau)$ and $A$ endowed with the discrete topology and where $\mathrm{B}$ is any basis of the topology $\tau$ consisting of ideals $\mathfrak{r}$ of $\mathcal{R}$, which is a restricted strict pro-ring if and only if $(\mathcal{R},\tau)$ is second-countable.  

The two associations $G\mapsto (\mathcal{R}_G,\tau_G)$ and $(\mathcal{R},\tau)\mapsto G_{(\mathcal{R},\tau)}$ are mutual inverses, and it is routine to verify that they give rise to an equivalence of categories between the category of restricted strict pro-rings and the category $(\mathrm{CRTop})$ whose objects 
are second-countable separated and
complete linearly topologized topological rings and whose morphisms are continuous homomorphisms of rings.

\medskip
In short, the theory of restricted strict affine ind-schemes is (anti-)equivalent to that of second-countable separated and complete linearly
topologized topological rings. In this setting, the theory of \emph{pro-affine algebras} over a base field $k$ pioneered by Kambayashi  \cite{Kam96,Kam03,Kam04} corresponds to that of restricted strict affine ind-schemes over $k$. 

\medskip
Throughout the rest of this article, we adopt the "algebraic-topological" viewpoint, in which we develop and formulate our results, leaving it largely up to the interested reader to retranslate the main statements into more "geometric" terms for the corresponding restricted strict affine ind-schemes.  
\subsection{Basic definitions, notations and examples}
We briefly summarize some essential definitions and properties of topological rings that will be used later in this article. We refer the reader to the Appendix for more detailed definitions and information. 

Throughout the rest of the article, the term \emph{topological ring} will always refer to a commutative topological ring $\mathcal{A}$  with unity, endowed with a linear topology with respect to the underlying topological abelian group structure for which there exists a fundamental system of open neighborhoods of $0_{\mathcal{A}}$
consisting of a countable family $(\mathfrak{a}_{n})_{n\in\mathbb{N}}$
of ideals of $\mathcal{A}$. In particular, every topological ring in this sense is a second-countable uniform topological space.  The term \emph{homomorphism of topological rings} is used to refer to a homomorphism of rings $f:\mathcal{A}\to \mathcal{B}$ which is continuous with respect to the topologies on $\mathcal{A}$ and $\mathcal{B}$. We denote by $\mathrm{CHom}(\mathcal{A},\mathcal{B})\subset \mathrm{Hom}(\mathcal{A},\mathcal{B})$ the subset of homomorphisms of topological rings.

We call a topological ring $\mathcal{A}$ \emph{complete} if it is separated as a topological space and complete 
with respect to its uniform structure. This holds if and only if given any fundamental system of open neighborhoods $(\mathfrak{a}_{n})_{n\in\mathbb{N}}$ of $0_{\mathcal{A}}$, the canonical homomorphism of rings $c:\mathcal{A}\to \varprojlim_{n\in \mathbb{N}} \mathcal{A}/\mathfrak{a}_n$, called the \emph{separated completion homomorphism} of $\mathcal{A}$, is an isomorphism. It follows in turn that if $\mathcal{A}$ and $\mathcal{B}$ are complete topological rings and $(\mathfrak{a} _{n})_{n\in\mathbb{N}}$ 
and $(\mathfrak{b}_{n})_{n\in\mathbb{N}}$
are any choice of fundamental systems of open neighborhoods of $0_{\mathcal{A}}$ and $0_{\mathcal{B}}$ respectively,
we have $$\mathrm{CHom}(\mathcal{A},\mathcal{B})=\varprojlim_{m\in \mathbb{N}}( \varinjlim_{n \in \mathbb{N}} \mathrm{Hom} (\mathcal{A}/\mathfrak{a}_n, \mathcal{B}/\mathfrak{b_m})).$$
Keeping the notation introduced in the previous subsection, we let $(\mathrm{CRTop})$ be the category whose objects are complete topological rings and whose morphisms are homomorphisms of topological rings.

\medskip

We now introduce a class of complete topological rings which will be used repeatedly.

\begin{example}[{Pro-polynomial rings}]\label{ex:pro-po} A \emph{pro-polynomial ring} over $\mathbb{Z}$ is a complete topological $\mathbb{Z}$-algebra $\mathcal{P}$ which admits a fundamental system $(\mathfrak{p}_n)_{n\in \mathbb{N}}$ of open neighborhoods of $0_{\mathcal{P}}$ consisting of ideals $\mathfrak{p}_n$ such that for every $n$, the ring $\mathcal{P}/\mathfrak{p}_n$ is a $\mathbb{Z}$-algebra isomorphic to a polynomial ring 
and there exists $m>n$ such that $\mathfrak{p}_m\subset \mathfrak{p}_n$ and the induced surjective morphism $\mathcal{P}/\mathfrak{p}_m\to \mathcal{P}/\mathfrak{p}_n$ is a coordinate projection.  
So $\mathcal{P}$ is either a polynomial ring $\mathbb{Z}[T_1,\ldots,T_r]$ in finitely many indeterminates endowed with the discrete topology or is isomorphic as a topological $\mathbb{Z}$-algebra to the separated completion 
  $$\widehat{\mathbb{Z}[(T_i)_{i\in \mathbb{N}}]}:=\varprojlim_{n\in \mathbb{N}}(\mathbb{Z}[(T_i)_{i\in \mathbb{N}}]/(T_i)_{i\geq n})$$
  of the polynomial ring $\mathbb{Z}[(T_i)_{i\in \mathbb{N}}]$ in countably many indeterminates $(T_i)_{i\in \mathbb{N}}$ with respect to the topology induced by the fundamental system of open ideals
  $(T_i)_{i\geq n}\mathbb{Z}[(T_i)_{i\in \mathbb{N}}]$. 

  A pro-polynomial ring with coefficients in a complete topological ring $\mathcal{A}$ is a complete topological $\mathcal{A}$-algebra isomorphic to the completed tensor product 
  $\mathcal{A}\widehat{\otimes}_{\mathbb{Z}}\mathcal{P}$, that is, to the separated completion of the usual tensor product $\mathcal{A}\otimes_{\mathbb{Z}} \mathcal{P}$ with respect to the linear topology
generated by open neighborhoods of $0$ of the form 
$\mathfrak{a}_n\otimes \mathcal{P}+\mathcal{A}\otimes \mathfrak{p}_n$, where $\mathfrak{a}_n$ and $\mathfrak{p}_n$ run respectively through the ideals in a fundamental system of open neighborhoods of $0_{\mathcal{A}}$ and $0_{\mathcal{P}}$.  

In the case where $\mathcal{P}$ is a polynomial ring $\mathbb{Z}[T_1,\ldots,T_r]$ in finitely many indeterminates endowed with the discrete topology, the resulting topological ring $\mathcal{A}\widehat{\otimes}_{\mathbb{Z}}\mathcal{P}$ is nothing but the ring 
$$\mathcal{A}\{T_1,\ldots, T_r\}=\varprojlim_{n\in \mathbb{N}} ((\mathcal{A}/\mathfrak{a}_n)[T_1,\ldots, T_r])$$
of \emph{restricted power series} 
with coefficients in $\mathcal{A}$. As a ring, $\mathcal{A}\{T_1,\ldots, T_r\}$ can be identified with the sub-$\mathcal{A}$-algebra of the algebra of formal power series $\mathcal{A}[[T_1,\ldots, T_n]]$ with coefficients in $\mathcal{A}$
consisting of formal power series 
\[
\sum_{I=(i_{1},\ldots,i_{r})\in\mathbb{N}^{r}}a_{I}T_{1}^{i_{1}}\cdots T_{r}^{i_{r}}
\]
such that the family $(a_{I})_{I\in\mathbb{N}^{r}}$ converges to
$0_{\mathcal{A}}$ for the topology on $\mathcal{A}$. We refer the reader to Section \ref{ref:Restricted power series} in the Appendix for details and a discussion of  several additional properties of these topological algebras. 

In the case where $\mathcal{P}\cong \widehat{\mathbb{Z}[(T_i)_{i\in \mathbb{N}}]}$ 
the topological ring 
$$\mathcal{A}\widehat{\otimes}_\mathbb{Z}\mathcal{P}\cong \varprojlim_{n\in \mathbb{N}}(\mathcal{A}[(T_i)_{i\in \mathbb{N}}]/(\mathfrak{a}_n\mathcal{A}[(T_i)_{i\in \mathbb{N}}]+(T_i)_{i\geq n}\mathcal{A}[(T_i)_{i\in \mathbb{N}}]))$$ 
is isomorphic to the limit $\varprojlim_{n\in \mathbb{N}} \mathcal{B}_n$ of the inverse system of restricted power series rings $\mathcal{B}_n=\mathcal{A}\{(T_i)_{i\leq n}\}$ 
with surjective  continuous projections $p_{m,n}\colon \mathcal{B}_m \rightarrow
  \mathcal{B}_n$ with kernels
  $(T_{n+1},\ldots, T_{m})\mathcal{B}_m$ when $m>n$ and $0_{\mathcal{B}_m}$ when $m=n$, endowed with the limit topology. 
  \end{example}

\section{Restricted exponential homomorphisms and topologically integrable derivations}\label{sub:rest-exp-hom} \label{sub:exp-basic}\label{sec:2} 
We introduce the notion of restricted exponential homomorphisms for complete topological rings and then describe a one-to-one correspondence between such homomorphisms and suitable systems of continuous iterated higher derivations which extends the classical correspondence  between algebraic exponential homomorphisms and locally finite iterative higher derivations of the ring \cite{M68,CML05,F06}.

\subsection{Restricted exponential homomorphisms}
Recall that the ring $\mathbb{Z}[T]$, where $T$ is an indeterminate, carries the structure of a cocommutative Hopf algebra whose comultiplication, coinverse and counit are given respectively by the following $\mathbb{Z}$-algebra homomorphisms: 
\begin{align*}
\begin{array}{lcl}
m\colon \mathbb{Z}[T]\rightarrow \mathbb{Z}[T]\otimes_{\mathbb{Z}} \mathbb{Z}[T] \cong \mathbb{Z}[T,T'], & & T \mapsto T+T' \\
\iota\colon \mathbb{Z}[T] \rightarrow \mathbb{Z}[T],  & & T\mapsto -T \\
\epsilon\colon \mathbb{Z}[T]\rightarrow \mathbb{Z}, & & T\mapsto 0.
\end{array}
\end{align*}
Given any complete topological ring $\mathcal{A}$, the complete topological ring $\mathcal{A}\{T\}=\mathcal{A}\widehat{\otimes} \mathbb{Z}[T]$ of restricted power series with coefficients in $\mathcal{A}$ inherits the structure of a cocommutative topological Hopf $\mathcal{A}$-algebra with comultiplication $\mathrm{id}_{\mathcal{A}}\widehat{\otimes}m$, coinverse $\mathrm{id}_{\mathcal{A}}\widehat{\otimes} \iota$ and counit $\mathrm{id}_{\mathcal{A}}\widehat{\otimes} \epsilon$. 

\begin{defn}
\label{def:Restricted-Exp-Map}Let $\mathcal{A}$ be a complete topological
ring and let $\mathcal{B}$ be a complete topological $\mathcal{A}$-algebra. A \emph{restricted exponential $\mathcal{A}$-homomorphism} is a homomorphism of topological $\mathcal{A}$-algebras  $$\mathrm{e}\colon\mathcal{B}\rightarrow\mathcal{B}\{T\}=\mathcal{B}\widehat{\otimes}_{\mathbb{Z}}\mathbb{Z}[T]$$ which defines a coaction of the Hopf $\mathcal{A}$-algebra $\mathcal{A}\{T\}$ on $\mathcal{B}$. This means equivalently that the following diagrams of homomorphisms of topological $\mathcal{A}$-algebras are commutative:

\begin{align*}
\begin{array}{ccc}
\xymatrix{  \mathcal{B} \ar[r]^{\mathrm{e}} \ar[d]_{\mathrm{e}} & \mathcal{B}\{T\} \ar[d]^{\mathrm{id}_{\mathcal{B}}\widehat{\otimes} m}  \\  \mathcal{B}\{T\} \ar[r]^-{\mathrm{e}\widehat{\otimes}\mathrm{id}_{\mathbb{Z}[T]}} & \mathcal{B}\{T'\}\{T\}= \mathcal{B}\{T',T \} } & &  \xymatrix{ \mathcal{B} \ar[r]^{\mathrm{e}} \ar[dr]^{\mathrm{id}_{\mathcal{B}}} & \mathcal{B}\{T\} \ar[d]^{q=\mathrm{id}_{\mathcal{B}}\widehat{\otimes} \epsilon} \\ & \mathcal{B}= \mathcal{B}\{T\}/T \mathcal{B}\{T\}. } 
\end{array}
\end{align*}
\end{defn}

\medskip

Let $p_i\colon \mathcal{B}[[T]]=\prod_{i\in \mathbb{N}} \mathcal{B}\rightarrow \mathcal{B}$ denote the $i$-th projection. Then by definition of the topology on $\mathcal{B}\{T\}$ viewed as a sub-$\mathcal{A}$-algebra of $\mathcal{B}[[T]]$, the composition $\mathrm{e}_i=p_i\circ \mathrm{e}\colon \mathcal{B}\rightarrow \mathcal{B}$ is a homomorphism of topological $\mathcal{A}$-modules for every $i\in \mathbb{N}$ for which we can write $$\mathrm{e}=\sum_{i\in \mathbb{N}} \mathrm{e}_i T^i.$$
The commutativity of the right hand side diagram of Definition \ref{def:Restricted-Exp-Map} means that $\mathrm{e}_0=\mathrm{id}_{\mathcal{B}}$. On the other hand, with the identifications made, the homomorphisms $\mathrm{e}\widehat{\otimes}\mathrm{id}_{\mathbb{Z}[T]}$ and $\mathrm{id}_{\mathcal{B}}\widehat{\otimes}m$ are given by 
\begin{align*}
\mathrm{e}\widehat{\otimes}\mathrm{id}_{\mathbb{Z}[T]} & \colon\mathcal{B}\{T\}\rightarrow\mathcal{B}\{T',T\},\;\sum_{i \in \mathbb{N}}b_{i}T^{i}\mapsto\sum_{i \in \mathbb{N}}\mathrm{e}(b_{i})T^{i}=\sum_{(i,j)\in \mathbb{N}^2} \mathrm{e}_j(b_i){T'}^{j}T^{i} \\
\mathrm{id}_{\mathcal{B}}\widehat{\otimes}m & \colon\mathcal{B}\{T\}\rightarrow\mathcal{B}\{T',T\},\;\sum_{i \in \mathbb{N}}b_{i}T^{i}\mapsto\sum_{i \in \mathbb{N}}b_{i}(T'+T)^{i}.
\end{align*}
The commutativity of the left hand side diagram in Definition \ref{def:Restricted-Exp-Map} says that in $\mathcal{B}\{T,T'\}$, we have
\begin{equation}\label{eq:rest-exp-condition} 
\sum_{(i,j)\in \mathbb{N}^2}(\mathrm{e}_j\circ \mathrm{e}_i){T'}^{j}T^{i} = \sum_{\ell \in \mathbb{N}}\mathrm{e}_{\ell}(T'+T)^{\ell}.
\end{equation}

\begin{example} \label{ex:chouette}
Let $\widehat{\mathbb{Z}[(X_i)_{i\in \mathbb{N}}]}=\varprojlim_{n\in \mathbb{N}}(\mathbb{Z}[(X_i)_{i\in \mathbb{N}}]/(X_i)_{i\geq n})$ be 
a pro-polynomial ring as in Example \ref{ex:pro-po}. Then the homomorphisms of topological $\mathbb{Z}$-algebras
$\epsilon_\pm\colon \mathbb{Z}[(X_i)_{i\in \mathbb{N}}]\to \widehat{\mathbb{Z}[(X_i)_{i\in \mathbb{N}}]} \{T\}$ given  by
 $$X_i\mapsto \epsilon_+(X_i)=\sum_{j\geq i}X_jT^{j-i}\quad \mbox{and}\quad X_i\mapsto \epsilon_-(X_i)=\sum_{j\leq i}\tbinom{i}{j}X_jT^{i-j},\quad i\in \mathbb{N}$$
 induce restricted exponential $\mathbb{Z}$-homomorphisms $\mathrm{e}_\pm=\widehat{\epsilon_\pm}\colon \widehat{\mathbb{Z}[(X_i)_{i\in \mathbb{N}}]}\to\widehat{\mathbb{Z}[(X_i)_{i\in \mathbb{N}}]}\{T\}$, see Proposition \ref{lem:Compl-Homo-Extension} for the construction of these induced homomorphisms.  
\end{example}

\subsection{Topologically integrable iterated higher derivations} 


\begin{defn}
\label{Def:IHD}\label{def:Continuous-Der} Let $\mathcal{A}$ be a topological ring and let $\mathcal{B}$ be a topological $\mathcal{A}$-algebra. A continuous \emph{iterated
higher $\mathcal{A}$-derivation} of $\mathcal{B}$ is a collection $D=\big\{ D^{(i)}\big\} _{i\geq0}$ of homomorphisms of topological $\mathcal{A}$-modules
$D^{(i)}\colon \mathcal{B}\rightarrow\mathcal{B}$ which satisfy the following properties:
\begin{enumerate}
\item The homomorphism $D^{(0)}$ is the identity homomorphism of $\mathcal{B}$,

\item For every $i\geq0$, the \emph{Leibniz rule} $D^{(i)}(bb')=\sum_{j=0}^{i}D^{(j)}(b)D^{(i-j)}(b')$
holds for every pair of elements $b,b'\in\mathcal{B}$,

\item For every $i,j\geq0$, $D^{(i)}\circ D^{(j)}=\binom{i+j}{i}D^{(i+j)}$.
\end{enumerate}
\end{defn}

The first two properties imply in particular that
$\partial=D^{(1)}\colon \mathcal{B}\rightarrow\mathcal{B}$ is a
continuous $\mathcal{A}$-derivation of $\mathcal{B}$ into itself. If
$\mathcal{A}$ contains the field $\mathbb{Q}$ then the third property
implies that $D^{(i)}=\tfrac{1}{i!}\partial^{i}$ for every $i\geq0$,
where $\partial^i$ denotes the $i$-th iterate of $\partial$. In this
case, a continuous iterated higher $\mathcal{A}$-derivation is then
uniquely determined by a continuous $\mathcal{A}$-derivation
$\partial$ of $\mathcal{B}$ into itself. The notion of higher
derivation was first introduced by Hasse and Schmidt in \cite{HS37}.

\begin{defn} \label{def:Pro-LND}Let $\mathcal{A}$ be a topological ring and let $\mathcal{B}$ be a topological $\mathcal{A}$-algebra. A
  \emph{topologically integrable iterated higher $\mathcal{A}$-derivation} of $\mathcal{B}$ (an
  $\mathcal{A}$-TIIHD for short) is a continuous iterated higher $\mathcal{A}$-derivation $D=\left\{ D^{(i)}\right\} _{i\geq0}$ such
  that the collection of homomorphisms of topological $\mathcal{A}$-modules $(D^{(i)})_{i\geq 0}$ is uniformly equicontinuous and pointwise convergent to the zero endomorphism of $\mathcal{B}$ (see Definition \ref{continuous-convergence} ) .

  When $\mathcal{A}$ contains the field $\mathbb{Q}$, we say that a
  continuous $\mathcal{A}$-derivation $\partial$ is
  \emph{topologically integrable} if its associated continuous
  iterated higher $\mathcal{A}$-derivation
  $D=\left\{ \tfrac{1}{i!} \partial^{i}\right\} _{i\geq0}$ is 
  topologically integrable. 
\end{defn}

\begin{rem} When $\mathcal{A}$ and $\mathcal{B}$ are topological rings endowed with the discrete topology, the condition that the family $(D^{(i)})_{i\geq 0}$ is uniformly equicontiunous is automatically satisfied. On the other hand, pointwise convergence to the zero map means equivalently that for every element $b\in B$ there exists an integer $i_{0}$ such
that $D^{(i)}(b)=0$ for every $i\geq i_{0}$. We thus recover in this case the classical notions of locally finite iterated higher
$\mathcal{A}$-derivation of $\mathcal{B}$ (\cite{CML05,M68}) and locally nilpotent $\mathcal{A}$-derivation of $\mathcal{B}$ (\cite{F06}).
\end{rem}

\begin{lem} \label{lem:Der-Comp-Extension} \label{lem:Pro-LND-CompExt} With the notation of Definition \ref{Def:IHD}, let $c\colon \mathcal{B}\rightarrow\widehat{\mathcal{B}}$ be the separated
completion of $\mathcal{B}$. Then for every continuous iterated higher $\mathcal{A}$-derivation $D=\big\{ D^{(i)}\big\} _{i\geq0}$ of $\mathcal{B}$, 
there exists a unique continuous iterated higher $\mathcal{A}$-derivation $\widehat{D}=\big\{ \widehat{D}^{(i)}\big\} _{i\geq0}$ of $\widehat{\mathcal{B}}$ such that $\widehat{D}^{(i)}\circ c=c\circ D^{(i)}$ for every $i\geq 0$. 

Furthermore, if $D$ is topologically integrable, then so is $\widehat{D}$. 
\end{lem}

\begin{proof} The existence of a unique collection of homomorphism of topological $\mathcal{A}$-modules $\widehat{D}^{(i)}$ such that $\widehat{D}^{(i)}\circ c=c\circ D^{(i)}$ for every $i\geq 0$ follows from the universal property of the separated completion, see Proposition \ref{lem:Compl-Homo-Extension}. Moreover, by construction, the $\widehat{D}^{(i)}$ satisfy the three identities in Definition \ref{Def:IHD} in restriction to the subset $c(\mathcal{B})$ of $\widehat{\mathcal{B}}$. Since the latter in dense in $\widehat{\mathcal{B}}$ and $\widehat{\mathcal{B}}$ is separated, they also hold on the whole of $\widehat{\mathcal{B}}$. The second assertion follows from Proposition \ref{pro:Uniform_extention}. 
\end{proof}

\subsection{The correspondence }

\begin{thm} \label{top-int-1-1-exp}
Let $\mathcal{A}$ be complete topological ring and let $\mathcal{B}$ be a complete topological $\mathcal{A}$-algebra. Then the map 
which associates to a topologically integrable iterated higher $\mathcal{A}$-derivation $D=\big\{D^{(i)}\big\}_{i\geq 0}$ of $\mathcal{B}$ the $\mathcal{A}$-module homomorphism

$$\mathrm{e}_D=\exp(TD):=\sum_{i\geq 0} D^{(i)} T^i\colon \mathcal{B}\rightarrow \mathcal{B}\{T\}, \; b\mapsto \sum_{i\geq 0} D^{(i)}(b)T^i$$ is well-defined 
and induces a one-to-one correspondence between 
topologically integrable iterated higher $\mathcal{A}$-derivations of $\mathcal{B}$ and restricted exponential $\mathcal{A}$-homomorphisms $\mathcal{B}\rightarrow\mathcal{B}\{T\}$.
\end{thm}

\begin{proof}
Since each $D^{(i)}:\mathcal{B}\to \mathcal{B}$ is a homomorphism of topological $\mathcal{A}$-modules, it follows from the universal property of the product topology that there exists a unique homomorphism of topological $\mathcal{A}$-algebras $\epsilon_D:\mathcal{B}\to \prod_{i\geq 0}\mathcal{B}=\mathcal{B}[[T]]$ such that $D^{(i)}=p_i\circ\epsilon_D$. By Lemma \ref{lem:cc-converge-to-restricted-map}, the condition that the family $(D^{(i)})_{i\geq 0}$ uniformly continuous and pointwise convergent to the zero homomorphism is exactly guaranteeing that $\epsilon_D$ factorizes through a homomorphism of topological $\mathcal{A}$-algebras $\mathrm{e}_D:\mathcal{B}\to \mathcal{B}\{T\}$. The additional properties of $D=\{D_i\}_{i\geq 0}$ listed in Definition \ref{Def:IHD} are then precisely those which express the commutativity of the two diagrams in Definition \ref{def:Restricted-Exp-Map}, showing that $\mathrm{e}_D$ is a restricted exponential homomorphism. Conversely, for a restricted exponential homomorphism $\mathrm{e}=\sum_{i\geq 0} \mathrm{e}_iT^i\colon \mathcal{B}\rightarrow \mathcal{B}\{T\}$, it follows from the identities \eqref{eq:rest-exp-condition} expressing the commutativity of the two diagrams in Definition \ref{def:Restricted-Exp-Map} that the collection of homomorphisms of topological $\mathcal{A}$-modules $D^{(i)}=\mathrm{e}_i$, $i\geq 0$, is a continuous iterated higher $\mathcal{A}$-derivation of $\mathcal{B}$. Since $e$ is continuous, it follows in turn from the definition of $\mathcal{B}\{T\}$ as a sub-$\mathcal{A}$-algebra of $\mathcal{B}[[T]]$ and its topology that the so-defined family $(D^{(i)})_{i\geq 0}$ is uniformly equicontinuous and pointwise convergent to the zero homomorphism.
\end{proof}

\begin{rem}
    In the case where the base topological ring $\mathcal{A}$ contains the field $\mathbb{Q}$, the fact that  a continuous iterated higher $\mathcal{A}$-derivation $D=\big\{D^{(i)}\big\}_{i\geq 0}$, is uniquely determined by the continuous $\mathcal{A}$-derivation $\partial=D^{(1)}$ of $\mathcal{B}$ implies that a restricted exponential $\mathcal{A}$-homomorphism $\mathrm{e}=\sum_{i\geq 0}\mathrm{e}_i T^i\colon \mathcal{B}\rightarrow \mathcal{B}\{T\}$ is uniquely determined by the topologically integrable $\mathcal{A}$-derivation $$\partial=\mathrm{e}_1=\tfrac{\partial}{\partial T}|_{T=0}\circ \mathrm{e}.$$ 
    
\end{rem}

\begin{example}
\label{exa:Standard-PLFIHD-RPS}Let $\mathcal{A}$ be a complete topological ring and let $\mathcal{B}=\mathcal{A}\{S\}$. The comultiplication $$m:\mathcal{B}\to \mathcal{B}\{T\}=\mathcal{A}\{S,T\}, \; S\mapsto S+T$$ of the Hopf $\mathcal{A}$-algebra $\mathcal{B}$ is a restricted exponential homomorphism whose corresponding topologically integrable iterated higher $\mathcal{A}$-derivation is given by the collection of homomorphisms
\[
D^{(i)}=(\frac{1}{i!}\frac{\partial^{i}}{\partial T^{i}}|_{T=0})\circ m\colon \mathcal{B}\rightarrow\mathcal{B}
\]
associating to an element $P(S)\in \mathcal{B}$ the $i$-th coefficient of the Taylor expansion
at $0$ of $P(S+T)\in \mathcal{B}\{T\}$ with respect to the variable $T$. In particular, $D^{(1)}$ is simply the $\mathcal{A}$-derivation $\tfrac{\partial}{\partial S}$
of $\mathcal{B}$. 
\end{example}

The following example illustrates the importance of uniform equicontinuity in the correspondence between topologically integrable $\mathcal{A}$-derivation and restricted exponential homomorphisms.

\begin{example}\label{ex:non-contconv-der}
Let $\mathcal{B}$ the 
separated completion of the polynomial ring $\mathbb{Q}[(X_i)_{i\in \mathbb{N}}]$ in countably many indeterminates $X_i$ with respect to the linear topology generated by the open ideal $\mathfrak{q}_n=(X_i)_{i\geq n}\mathbb{Q}[(X_i)_{i\in \mathbb{N}}]$ of $0$. 
The $\mathbb{Q}$-derivation $\partial$  of $\mathbb{Q}[(X_i)_{i
  \in \mathbb{N}}]$ defined by $$\partial(X_0)=X_1,\quad
\partial(X_{2i-1})=X_{2i+1}, \textrm{ and}\quad \partial(X_{2i})=X_{2i-2} \quad \forall i\geq 1.$$ is continuous and pointwise convergent to $0$. But the collection of homomorphisms of topological $\mathbb{Q}$-algebras $(\frac{1}{i!}\partial^i)_{i\geq 0}$ is not uniformly equicontinuous.   Indeed, since $\partial^{\ell}(X_{2\ell})=X_0$,
it follows that for any given  $i\geq 0$ there cannot exist any
integer $n_0$ such that
$\partial^{n}(\mathfrak{q}_{i})\subseteq
\mathfrak{q}_{1}$ for every $n\geq n_0$. This implies in turn that the associated
$\mathbb{Q}$-derivation $\widehat{\partial}$ of the pro-polynomial $\mathbb{Q}$-algebra $\mathcal{B}$ is not
topologically integrable. Here, since the collection $(\frac{1}{i!}\widehat{\partial}^i)_{i\geq 0}$ is pointwise convergent to $0$, the exponential homomorphism
$\exp(T\widehat{\partial})\colon \mathcal{B}\rightarrow \mathcal{B}\{T\}$ is well defined as a ring homomorphism, but it is not continuous.
\end{example}

\section{Basic properties of restricted exponential homomorphisms} \label{sec:3}
In the next subsections, we establish general properties of restricted exponential homomorphisms. We leave to the reader the pleasure to reformulate, through the correspondence given in Theorem  \ref{top-int-1-1-exp}, each of these results in the equivalent language of topologically integrable higher iterated derivations, see nevertheless subsection \ref{subsec:top-der-summary} for a summary in the topologically integrable derivation setting.  

\subsection{Rings of invariants and associated restricted exponential homomorphisms}

\begin{defn} Let $\mathcal{B}$ be a complete topological ring and let $\mathrm{e}=\sum_{i\geq 0}\mathrm{e}_i T^i\colon\mathcal{B}\rightarrow \mathcal{B}\{T\}$ be a restricted exponential homomorphism. We say that en element $b\in\mathcal{B}$ is $\mathrm{e}$-\emph{invariant} if $\mathrm{e}(b)=i_{0}(b)$, where $i_0:\mathcal{B}\to\mathcal{B}\{T\}$ denotes the inclusion of the subring of constant restricted power series. We denote by $$\mathcal{B}^{\mathrm{e}}=\operatorname{Ker}(\mathrm{e}-i_{0})=\bigcap_{i\geq 1}\mathrm{Ker}(\mathrm{e}_i)\subseteq \mathcal{B} $$ the subset of all $\mathrm{e}$-invariant elements of $\mathcal{B}$, endowed with the induced topology.
\end{defn}

\begin{example}

    Let $\mathrm{e}\colon\mathcal{B}\to\mathcal{B}\{T\}$ be a restricted exponential homomorphism. Then, for every element $b\in \mathcal{B}$ such that $\mathrm{e}(b)$ is a non-constant polynomial, it follows from the identity \eqref{eq:rest-exp-condition} that the leading coefficient of $\mathrm{e}(b)$ belongs to $\mathcal{B}^{\mathrm{e}}$. In Example~\ref{ex:chouette}, the image by $\mathrm{e}_+$ of every element in  $\mathcal{B}$ not in the image of the canonically induced injective homomorphism $i\colon\mathbb{Z}\to\mathcal{B}$ is a restricted power series which is not a polynomial. One can actually check further that $\mathcal{B}^{\mathrm{e}_+}=i(\mathbb{Z})$, see also Example \ref{exa:translation}. In contrast, in the same example $\mathcal{B}^{\mathrm{e}_-}$ contains in particular the image of $\mathbb{Z}[X_0]$.

\end{example}

\begin{prop}\label{lem:kernel} Let $\mathcal{B}$ be a complete topological ring and let $\mathrm{e}=\sum_{i\geq 0} \mathrm{e}_i T^i\colon \mathcal{B}\rightarrow \mathcal{B}\{T\}$ be a restricted exponential homomorphism. Then the following hold:

a) The set $\mathcal{B}^{\mathrm{e}}$ is a complete topological subring of $\mathcal{B}$.

b) For every $i\geq 1$, the homomorphism $\mathrm{e}_i\colon \mathcal{B}\rightarrow \mathcal{B}$ is a homomorphism of topological $\mathcal{B}^{\mathrm{e}}$-modules. 

c) If $\mathcal{B}$ admits a fundamental system $(\mathfrak{b}_{n})_{n\in\mathbb{N}}$ of open ideals consisting of \emph{prime} ideals of $\mathcal{B}$ then $\mathcal{B}^{\mathrm{e}}$ is factorially closed in $\mathcal{B}$. In particular, every invertible element of $\mathcal{B}$ is contained in $\mathcal{B}^{\mathrm{e}}$.
\end{prop}

\begin{proof} The fact that $\mathcal{B}^{\mathrm{e}}$ is a subgring of $\mathcal{B}$ is clear. Since $\mathcal{B}\{T\}$ is complete, hence separated, $\{0\}$ is a closed subset of $\mathcal{B}\{T\}$. Since $\mathrm{e}-i_{0}\colon \mathcal{B}\rightarrow\mathcal{B}\{T\}$ is a homomorphism of topological groups, $\mathcal{B}^{\mathrm{e}}$ is  a closed subgroup of $\mathcal{B}$, hence a complete topological group since $\mathcal{B}$ is complete. Assertion b) is clear from the definition of  $\mathcal{B}^{\mathrm{e}}$ and the homomorphisms $\mathrm{e}_i$. 

Now let $(\mathfrak{b}_{n})_{n\in\mathbb{N}}$ be fundamental system of open ideal of $\mathcal{B}$ consisting of prime ideals of $\mathcal{B}$ and let 
\[
\pi_{n}\colon\mathcal{B}\{T\}=\varprojlim_{n\in\mathbb{N}}(\mathcal{B}/\mathfrak{b}_{n})[T]\rightarrow(\mathcal{B}/\mathfrak{b}_{n})[T],\quad n\in\mathbb{N},
\]
be the canonical projections.  Let $b,b'\in \mathcal{B}$ such that $\mathrm{e}(bb')=\mathrm{e}(b)\mathrm{e}(b')=bb'$. For every $n\geq0$, we have 
\[
\pi_{n}(\mathrm{e}(b)\mathrm{e}(b'))=(\sum_{i \in \mathbb{N}}\pi_{n}(\mathrm{e}_{i}(b))T^{i})(\sum_{i \in \mathbb{N}}\pi_{n}(\mathrm{e}_i(b'))T^{i})=\pi_{n}(bb')=\pi_{n}(b)\pi_{n}(b')
\]
in the integral domain $(\mathcal{B}/\mathfrak{b}_{n})[T]$. It follows that $\pi_{n}(\mathrm{e}_{i}(b))=\pi_{n}(\mathrm{e}_{i}(b'))=0$ for every $i\geq1$. This implies that for every $i\geq1$, $\mathrm{e}_{i}(b)$ and $\mathrm{e}_i(b')$ belong to $\bigcap_{n\geq1}\mathfrak{b}_{n}=\{0\}$ as $\mathcal{B}$ is separated. Thus $\mathrm{e}(b)=b$ and $\mathrm{e}(b')=b'$. Finally, if $b\in\mathcal{B}$ is invertible, then $bb^{-1}=1\in\mathcal{B}^{\mathrm{e}}$ and so, $b$ and $b^{-1}$ belong to $\mathcal{B}^{\mathrm{e}}$.
\end{proof}

\begin{example} \label{ex:non-fc-kernel} If the topology on $\mathcal{B}$ is the discrete one, the existence of a fundamental system of open prime ideals of $\mathcal{B}$ is equivalent to the property that $\mathcal{B}$ is an integral domain. This is no longer true in general, and the conclusion of assertion c) in Proposition \ref{lem:kernel} does not hold under the weaker assumption that $\mathcal{B}$ is integral. Indeed, let $\mathcal{B}=\varprojlim_{n\in\mathbb{N}}\mathbb{C}[u]/(u^{n})\cong\mathbb{C}[[u]]$
be the completion of $\mathbb{C}[u]$ for the $u$-adic topology. The homomorphism of $\mathbb{C}$-algebras $\mathbb{C}[u]\rightarrow \mathbb{C}[[u]]\{T\}$
defined by  \[
u\mapsto u\sum_{i \in \mathbb{N}}(uT)^{i}=\sum_{i \in \mathbb{N}}u^{i+1}T^{i}
\] 
induces a uniquely determined continuous homomorphism $\mathrm{e}\colon \mathcal{B}\rightarrow\mathcal{B}\{T\}$
which satisfies the axioms of a restricted exponential
$\mathbb{C}$-homomorphism and whose ring of invariants $\mathcal{B}^e$ equals the subring of constant formal power series. In particular, the invertible element $1-u$ of $\mathcal{B}$ does not belong to $\mathcal{B}^{\mathrm{e}}$.
\end{example}

\begin{prop}\label{lem:replicat} Let $\mathcal{B}$ be a complete topological ring and let $\mathrm{e}=\sum_{i\geq 0} \mathrm{e}_iT^i\colon\mathcal{B}\rightarrow \mathcal{B}\{T\}$ be a restricted exponential homomorphism. Then for every element $a\in \mathcal{B}^{\mathrm{e}}$, the homomorphism $$\mathrm{e}_{\lambda(a)}:=\lambda(a)\circ\mathrm{e}\colon\mathcal{B} \stackrel{\mathrm{e}}{\rightarrow} \mathcal{B}\{T\} \stackrel{T\mapsto aT}{\longrightarrow} \mathcal{B}\{T\}$$  is a restricted exponential homomorphism. 
\end{prop}
\begin{proof} Since the homomorphisms $\mathrm{e}$ and $\lambda(a)$ are continuous, so is $\mathrm{e}_{\lambda(a)}$. The commutativity of the right hand side diagram in Definition \ref{def:Restricted-Exp-Map} for $\mathrm{e}_{\lambda(a)}$ is clear. We have $$\mathrm{e}_{\lambda(a)}=\sum_{i\geq 0} \mathrm{e}_i(aT)^i=\sum_{i \geq 0} (a^i\mathrm{e}_i)T^i,$$
where $a^i\mathrm{e}_i$ is the homomorphism defined by $b\mapsto a^i\mathrm{e}_i(b)$ for every $b\in \mathcal{B}$. Since $a\in \mathcal{B}^{\mathrm{e}}$, $a^i\in \mathcal{B}^{\mathrm{e}}$ for every $i\geq 0$. Since by  Proposition  \ref{lem:kernel} b) each $\mathrm{e}_i$, $i\geq 1$, is a homomorphism of topological $\mathcal{B}^{\mathrm{e}}$-module, applying \eqref{eq:rest-exp-condition}, we obtain 
\begin{align*}
\begin{array}{rcl}
 (\mathrm{e}_{\lambda(a)}\widehat{\otimes}\mathrm{id}_{\mathbb{Z}[T]})\circ \mathrm{e}_{\lambda(a)} &= &\sum_{i \in \mathbb{N}} (\mathrm{e}_{\lambda(a)}\circ (a^i\mathrm{e}_i))T^i \\
& = & \sum_{(i,j) \in \mathbb{N}^2}  (\mathrm{e}_j\circ(a^i\mathrm{e}_i))(aT')^jT^i \\
& = & \sum_{(i,j) \in \mathbb{N}^2}  (\mathrm{e}_j\circ\mathrm{e}_i)(aT')^j(aT)^i \\
& = & \sum _{\ell \in \mathbb{N}} \mathrm{e}_\ell (aT'+aT)^\ell= (\mathrm{id}_{\mathcal{B}}\widehat{\otimes}m)\circ \mathrm{e}_{\lambda(a)}.
\end{array}
\end{align*} 
This shows that the commutativity of left hand side diagram  in Definition \ref{def:Restricted-Exp-Map} is satisfied for $\mathrm{e}_{\lambda(a)}$. 
\end{proof}

\begin{prop}\label{prop:Localization-Kernel} Let $\mathcal{B}$ be a complete topological ring and let $\mathrm{e}=\sum_{i\geq 0}\mathrm{e}_iT^i\colon \mathcal{B} \rightarrow \mathcal{B}\{T\}$ be a restricted exponential homomorphism. Let $S\subset \mathcal{B}^{\mathrm{e}}$ be a multiplicatively closed subset, let $\widetilde{j}\colon\mathcal{B}\rightarrow \widehat{S^{-1}\mathcal{B}}$ be the separated completed localization homomorphism and let $\widetilde{j}_T\colon\mathcal{B}\{T\}\rightarrow \widehat{S^{-1}\mathcal{B}}\{T\}$ be the induced homomorphism. 

Then there exists a unique restricted exponential homomorphism $\widehat{S^{-1}\mathrm{e}}\colon\widehat{S^{-1}\mathcal{B}}\rightarrow\widehat{S^{-1}\mathcal{B}}\{T\}$
such that $\widetilde{j}_{T}\circ\mathrm{e}=\widehat{S^{-1}\mathrm{e}}\circ\widetilde{j}$
\end{prop}

\begin{proof}
We identify $S\subset \mathcal{B}$ with $i_0(S)\subset \mathcal{B}\{T\}$ and $\widehat{ S^{-1}(\mathcal{B}\{T\})}$ with $\widehat{S^{-1}\mathcal{B}}\{T\}$ by the canonical isomorphism of Lemma \ref{lem:Spe-Comp-Loc-Restricted-degree-0}.  Since $S\subset\mathcal{B}^{\mathrm{e}}$,we have $\mathrm{e}(S)=S\subset \mathcal{B}\{T\}$ so that by the universal property of separated completed localization, there exists a unique homomorphism of topological rings $$\widehat{S^{-1}\mathrm{e}}\colon\widehat{S^{-1}\mathcal{B}}\rightarrow \widehat{S^{-1}\mathcal{B}}\{T\}$$ such that $\widetilde{j}_{T}\circ\mathrm{e}=\widehat{S^{-1}\mathrm{e}}\circ\widetilde{j}$. Since $S\subset \mathcal{B}^{\mathrm{e}}$ and since $\mathrm{e}_i$ is a homomorphism of topological $\mathcal{B}^{\mathrm{e}}$-module for every $i\geq 0$ by Proposition \ref{lem:kernel} b), it follows that each $\mathrm{e}_i$, $i\in \mathbb{N}$, induces a uniquely determined homomorphism of topological $S^{-1}\mathcal{B}^{\mathrm{e}}$-modules $S^{-1}\mathrm{e}_i\colon S^{-1}\mathcal{B}\rightarrow S^{-1}\mathcal{B}$, hence by the universal property of separated completed localization, a homomorphism  $\widehat{S^{-1}\mathrm{e}_i}\colon \widehat{S^{-1}\mathcal{B}}\rightarrow \widehat{S^{-1}\mathcal{B}}$ of topological $\widehat{S^{-1}\mathcal{B}^{\mathrm{e}}}$-modules. By construction, we then have  $$\widehat{S^{-1}\mathrm{e}}=\sum_{i\in \mathbb{N}} \widehat{S^{-1}\mathrm{e}_i} T^i.$$
To show that $\widehat{S^{-1}\mathrm{e}}$ is a restricted exponential homomorphism, it is enough to check the commutativity of the two diagrams of Definition \ref{def:Restricted-Exp-Map} in restriction to the dense image of $S^{-1}\mathcal{B}$ in $\widehat{S^{-1}\mathcal{B}}$ by the separated completion morphism $c\colon S^{-1}\mathcal{B}\rightarrow \widehat{S^{-1}\mathcal{B}}$. 
Let $x=s^{-1} b\in S^{-1}\mathcal{B}$, where $b\in \mathcal{B}$. Then by definition of $\widehat{S^{-1}\mathrm{e}}$, we have $$\widehat{S^{-1}\mathrm{e}}(c(x))=\sum_{i\in \mathbb{N}}\widehat{S^{-1}\mathrm{e}_i}(c(x)) T^i= c_T\big(\sum_{i \in \mathbb{N}} (S^{-1}\mathrm{e}_i)(x)T^{i}\big)=c_T\big(\sum_{i \in \mathbb{N}} s^{-1}\mathrm{e}_i(b)T^{i}\big),$$ where $c_T\colon S^{-1}\mathcal{B}\{T\}\rightarrow \widehat{S^{-1}\mathcal{B}}\{T\}$ denotes the separated completion homomorphism. This immediately implies the commutativity of the right hand side diagram of  Definition \ref{def:Restricted-Exp-Map}. On the other hand, letting $c_{T,T'}\colon S^{-1}\mathcal{B}\{T,T'\}\rightarrow \widehat{S^{-1}\mathcal{B}}\{T,T'\}$ be the separated completion homomorphism, we have

\begin{align*}
\begin{array}{rcl}
(\widehat{S^{-1}\mathrm{e}}\widehat{\otimes}\mathrm{id}_{\mathbb{Z}[T]})(\widehat{S^{-1}\mathrm{e}}(c(x)) & = & \sum_{i\in\mathbb{N}}\widehat{S^{-1}\mathrm{e}}(c(s^{-1}\mathrm{e}_{i}(b)))T^{i} \\
& = & \sum_{(i,j)\in\mathbb{N}}\widehat{S^{-1}\mathrm{e}}_{j}(c(s^{-1}\mathrm{e}_{i}(b))){T'}^{j}T^{i} \\
& = & c_{T,T'}(\sum_{(i,j)\in\mathbb{N}^2}s^{-1}\mathrm{e}_{j}((s^{-1}\mathrm{e}_{i}(b))){T'}^{j}T^{i}) \\
& = & c_{T,T'}(\sum_{(i,j)\in\mathbb{N}^2}s^{-1}\mathrm{e}_{j}(\mathrm{e}_{i}(b)){T'}^{j}T^{i}) \\
& = & c_{T,T'}(\sum_{\ell\in\mathbb{N}}s^{-1}\mathrm{e}_{\ell}(b)(T'+T)^{\ell}) \\
& = & c_{T,T'}(\sum_{\ell\in\mathbb{N}}S^{-1}\mathrm{e}_{\ell}(x)(T'+T)^{\ell}) \\
& = & (\sum_{\ell\in\mathbb{N}}\widehat{S^{-1}\mathrm{e}_{\ell}}(c(x))(T'+T)^{\ell})=(\mathrm{id}_{\mathcal{B}}\widehat{\otimes}m)(\widehat{S^{-1}\mathrm{e}}(c(x)), 
\end{array}
\end{align*}
which shows the commutativity of the left hand side diagram. 
\end{proof}

\begin{prop}\label{prop:induced-quotient-exp} %
  Let $\pi\colon \mathcal{B}\rightarrow \mathcal{C}$ be a surjective open
  homomorphism of complete topological rings. Let
  $I=\operatorname{Ker}(\pi)$ and let
  $\mathrm{e}\colon \mathcal{B}\rightarrow \mathcal{B}\{T\}$ be a restricted
  exponential homomorphism. Assume that
  $\mathrm{e}(I)\subset i_0(I)\mathcal{B}\{T\}$. Then there exists a
  unique restricted exponential homomorphism
  $\overline{\mathrm{e}}\colon \mathcal{C}\rightarrow \mathcal{C}\{T\}$ such
  that $\pi_{T}\circ \mathrm{e}=\overline{\mathrm{e}}\circ\pi$, where
  $\pi_T\colon \mathcal{B}\{T\}\rightarrow \mathcal{C}\{T\}$ is the unique
  homomorphism of topological $\mathcal{B}$-algebras which maps $T$ to
  $T$.
\end{prop}

\begin{proof} The assumptions imply that $\mathrm{e}$ induces a unique homomorphism of rings $$\widetilde{\mathrm{e}}\colon  \mathcal{C}=\mathcal{B}/I\rightarrow 
\mathcal{B}\{T\}/i_0(I)\mathcal{B}\{T\}\cong \mathcal{B}/I\otimes_{\mathcal{B}} \mathcal{B}\{T\}\cong \mathcal{C}\otimes_{\mathcal{B}} \mathcal{B}\{T\} $$ such that $(\pi\otimes \mathrm{id}_{\mathcal{B}\{T\}})\circ \mathrm{e}=\widetilde{\mathrm{e}}\circ \pi$. Since $\pi$ is open, $\widetilde{\mathrm{e}}$ is continuous when we endow the ring $\mathcal{C}\otimes_{\mathcal{B}} \mathcal{B}\{T\}$ with the linear topology generated by open ideals of the form $U_{\mathcal{C}}\otimes \mathcal{B}\{T\}+\mathcal{C}\otimes V_{\mathcal{B}\{T\}}$ where $U_{\mathcal{C}}$ and $ V_{\mathcal{B}\{T\}}$ run through the sets of open ideals of $\mathcal{C}$ and $\mathcal{B}\{T\}$ respectively. Note also that $\pi\otimes \mathrm{id}_{\mathcal{B}\{T\}}$ is an open homomorphism of topological rings. 
The composition $\overline{\mathrm{e}}\colon  \mathcal{C}\rightarrow \mathcal{C}\widehat{\otimes}_{\mathcal{B}}\mathcal{B}\{T\}\cong \mathcal{C}\{T\}$ of $\widetilde{\mathrm{e}}$ with the separated completion homomorphism $c\colon  \mathcal{C}\otimes_{\mathcal{B}} \mathcal{B}\{T\} \rightarrow \mathcal{C}\widehat{\otimes}_{\mathcal{B}}\mathcal{B}\{T\}$ is then a homomorphism of topological rings. The commutativity of the diagrams in Definition  \ref{def:Restricted-Exp-Map} for $\bar{e}$ is easy to check. 
\end{proof}

\subsection{Operations on restricted exponential homomorphism}
\begin{prop} \label{lem:sum-restr-exp} Let $\mathcal{B}$ be a complete topological ring and let $\Delta \colon  \mathcal{B}\{T,T'\}\rightarrow \mathcal{B}\{T''\}$ denotes the unique homomorphism of topological $\mathcal{B}$-algebras which maps $T$ and $T'$ to $T''$. Let $\mathrm{e}\colon \mathcal{B}\rightarrow \mathcal{B}\{T\}$ and $\mathrm{e}'\colon \mathcal{B}\rightarrow \mathcal{B}\{T'\}$  be restricted exponential homomorphisms such that the following diagram commutes 

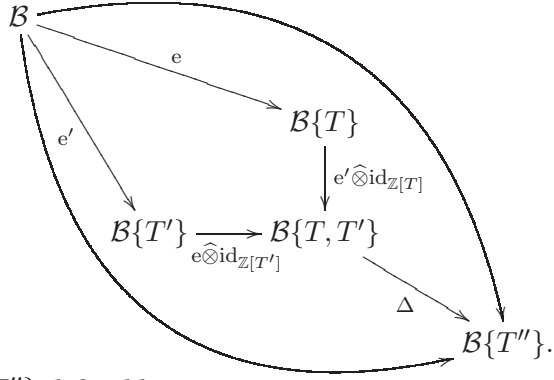
\[
\xymatrix{
\mathcal{B}\ar[ddr]_{\mathrm{e}'} \ar[rrd]^{\mathrm{e}} \ar@/_23mm/[dddrrr] \ar@/^20mm/[rrrddd]  \\
& &\mathcal{B}\{T\}\ar[d]^{\mathrm{e}'\widehat{\otimes}\mathrm{id}_{\mathbb{Z}[T]}}&\\
&\mathcal{B}\{T'\}\ar[r]_{\mathrm{e}\widehat{\otimes}\mathrm{id}_{\mathbb{Z}[T']}}& \mathcal{B}\{T,T'\} \ar[rd]_{\Delta}     \\
& & & \mathcal{B}\{T''\}.
}
\]
Then the map $\mathrm{e}''\colon \mathcal{B}\rightarrow \mathcal{B}\{T''\}$ defined by $$\mathrm{e}''= \Delta\circ(\mathrm{e}\widehat{\otimes}\mathrm{id}_{\mathbb{Z}[T']}) \circ \mathrm{e}'=\Delta\circ(\mathrm{e}'\widehat{\otimes}\mathrm{id}_{\mathbb{Z}[T]})\circ \mathrm{e}$$ is a restricted exponential homomorphism.
\end{prop}

\begin{proof} Being the composition of homomorphisms of topological
  rings, $\mathrm{e}''$ is a homomorphism of topological rings. Denoting
  $T''$ by $T_0$, we have, by definition of $\mathrm{e}''$,
\begin{equation}\label{eq:commutative}
\mathrm{e}''=\sum_{n\in \mathbb{N}} \mathrm{e}_n'' T_0^n=\sum_{(i,j)\in \mathbb{N}^2} (\mathrm{e}_j\circ \mathrm{e}_i')T_0^{i+j}.
\end{equation}
This implies in particular that
$\mathrm{e}_0''=\mathrm{e}_0'\circ
\mathrm{e}_0=\mathrm{id}_{\mathcal{B}}$, hence that the commutativity
of the right hand side diagram of Definition
\ref{def:Restricted-Exp-Map} holds for $\mathrm{e}''$.  Combining
Equation \eqref{eq:commutative} above with Equation
\eqref{eq:rest-exp-condition} for $\mathrm{e}$ and $\mathrm{e}'$, we
obtain on the other hand that
\begin{equation*}
\begin{array}{rcl}
  (\mathrm{id}_{\mathcal{B}}\widehat{\otimes}m)\circ \mathrm{e}''&= &\sum_{(i,j)\in\mathbb{N}^{2}}(\mathrm{e}_{j}'\circ \mathrm{e}_{i})(T_{0}+T_{0}')^{i+j}\\
                                                             &= &\sum_{(i,k,\ell)\in\mathbb{N}^{3}}(\mathrm{e}_{\ell}'\circ\mathrm{e}_{k}'\circ \mathrm{e}_{i}){T_{0}'}^{\ell}T_{0}^{k}(T_{0}+T_{0}')^{i}\\
                                                             &= & \sum_{(i,k,\ell)\in\mathbb{N}^{3}}(\mathrm{e}_{i}\circ \mathrm{e}_{\ell}'\circ \mathrm{e}_{k}')(T_{0}+T_{0}')^{i}{T_{0}'}T_{0}^{k}\\
                                                             &= &\sum_{(m,n,k,\ell)\in\mathbb{N}^{4}}(\mathrm{e}_{n}\circ\mathrm{e}_{m}\circ\mathrm{e}_{\ell}'\circ \mathrm{e}_{k}'){T_{0}'}^{n}T_{0}^{m}{T_{0}'}^{\ell}T_{0}^{k}\\
                                                             &= &\sum_{(m,n,k,\ell)\in\mathbb{N}^{4}}(\mathrm{e}_{n}\circ\mathrm{e}_{\ell}'\circ\mathrm{e}_{m}\circ\mathrm{e}_{k}'){T_{0}'}^{n+\ell}T_{0}^{k+m}\\
                                                             &= &\sum_{(i,n,\ell)\in\mathbb{N}^{3}}(\mathrm{e}_{n}\circ\mathrm{e}_{\ell}'\circ\mathrm{e}_{i}''){T_{0}'}^{n+\ell}T_{0}^{i}\\
                                                             &= &\sum_{(i,j)\in\mathbb{N}^{2}}(\mathrm{e}_{j}''\circ\mathrm{e}_{i}''){T_{0}'}^{j}T_{0}^{i}\\
                                                             &=&(\mathrm{e}''\widehat{\otimes}\mathrm{id}_{\mathbb{Z}[T_{0}]})\circ\mathrm{e}'',
\end{array}
\end{equation*}
which shows that $\mathrm{e}''$ is a restricted exponential homomorphism. 
\end{proof}

\begin{prop}\label{prop:proj-lim-restricted} Let $(\mathcal{B}_n)_{n\in \mathbb{N}}$ be a countable inverse system of complete topological rings with continuous and surjective transition homomorphisms $p_{m,n}\colon \mathcal{B}_m\rightarrow \mathcal{B}_n$  for every $m\geq n\geq 0$. Let $\mathcal{B}$ be its  limit and let $p_n\colon \mathcal{B}\rightarrow \mathcal{B}_n$, $n\in \mathbb{N}$, be the canonical continuous projections. Let $\mathrm{e}_n\colon \mathcal{B}_{n}\rightarrow  \mathcal{B}_{n}\{T\}$, $n\in \mathbb{N}$, be a collection of restricted exponential homomorphisms such that  $$\mathrm{e}_n\circ p_{m,n}=(p_{n,m}\widehat{\otimes}\mathrm{id}_{\mathbb{Z}[T]})\circ\mathrm{e}_n \qquad \forall m\geq n\geq 0.$$ Then there exists a unique restricted exponential homomorphism $\underline{\mathrm{e}}=\varprojlim_{n\in \mathbb{N}}\mathrm{e}_n\colon  \mathcal{B} \rightarrow \mathcal{B}\{T\}$ such that $\mathrm{e}_n\circ p_{n}=(p_{n}\widehat{\otimes}\mathrm{id}_{\mathbb{Z}[T]})\circ\underline{\mathrm{e}}$ for every $n\in \mathbb{N}$.
\end{prop}

\begin{proof} The hypothesis combined with Lemma  \ref{lem:proj-lim-restricted} implies the existence of a unique homomorphism of topological rings $$\underline{\mathrm{e}}=\varprojlim_{n\in \mathbb{N}}\mathrm{e}_n\colon \mathcal{B}=\varprojlim_{n\in \mathbb{N}} \mathcal{B}_n\rightarrow \varprojlim_{n\in \mathbb{N}} \mathcal{B}_n\{T\}\cong \mathcal{B}\{T\}$$ 
such that $\mathrm{e}_n\circ p_n=(p_n\widehat{\otimes}\mathrm{id}_{\mathbb{Z}[T]})\circ \underline{\mathrm{e}}$ for every $n\in \mathbb{N}$. The equality $\mathrm{e}_{n,0}=\mathrm{id}_{\mathcal{B}}$ for every $n\in \mathbb{N}$ implies that $\underline{\mathrm{e}}_0=\mathrm{id}_{\mathcal{B}}$. Similarly, since $(\mathrm{id}_{\mathcal{B}}\widehat{\otimes}m)\circ \mathrm{e}_n=(\mathrm{e}_n\widehat{\otimes} \mathrm{id}_{\mathrm{Z}[T]})\circ \mathrm{e}_n$ for every $n\in \mathbb{N}$, it follows that $(\mathrm{id}_{\mathcal{B}}\widehat{\otimes}m)\circ \underline{\mathrm{e}}=(\underline{\mathrm{e}}\widehat{\otimes} \mathrm{id}_{\mathrm{Z}[T]})\circ \underline{\mathrm{e}}$, showing that $\underline{\mathrm{e}}$ is a restricted exponential homomorphism.
\end{proof}

\subsection{Restricted exponential homomorphisms and automorphisms}

\begin{prop} Let $\mathcal{B}$ be a complete topological ring and let $\mathrm{e}\colon \mathcal{B}\rightarrow \mathcal{B}\{T\}$ be a restricted exponential homomorphism. Then for every automorphism of topological ring $\alpha$ of $\mathcal{B}$, the composition 
$$ {}^\alpha\!\mathrm{e}:= (\alpha \widehat{\otimes} \mathrm{id}_{\mathbb{Z}[T]})\circ\mathrm{e}\circ\alpha^{-1}\colon  \mathcal{B}\rightarrow \mathcal{B}\{T\}$$ 
is a restricted exponential homomorphism. 
\end{prop}

\begin{proof} It is clear that ${}^\alpha\!\mathrm{e}$ is homomorphism of topological rings. By definition, we have $${}^{\alpha}\!\mathrm{e}=\sum_{i\in \mathbb{N}} {}^{\alpha}\!\mathrm{e}_i T^i=\sum_{i\in \mathbb{N}} (\alpha\circ \mathrm{e}_i\circ \alpha^{-1}) T^i.$$ 
Since $\mathrm{e}_0=\mathrm{id}_{\mathcal{B}}$, we have  ${}^{\alpha}\!\mathrm{e}_0=\mathrm{id}_{\mathcal{B}}$ showing that commutativity of the right hand side diagram of Definition \ref{def:Restricted-Exp-Map} holds for ${}^{\alpha}\!\mathrm{e}(b)$. On the other hand, applying Equation \eqref{eq:rest-exp-condition}, we have

\begin{align*}
\begin{array}{rcl}
(\mathrm{id}_{\mathcal{B}}\widehat{\otimes}m)\circ {}^{\alpha}\!\mathrm{e} &	= &\sum_{\ell\in\mathbb{N}}(\alpha\circ \mathrm{e}_{\ell}\circ \alpha^{-1})(T+T')^{\ell} \\
	& =&\sum_{(i,j)\in\mathbb{N}^{2}}(\alpha\circ \mathrm{e}_{i}\circ \mathrm{e}_{j}\circ \alpha^{-1}){T'}^{j}T^{i} \\
	&=&\sum_{(i,j)\in\mathbb{N}^{2}}\big((\alpha\circ \mathrm{e}_{i}\circ \alpha^{-1})\circ (\alpha\circ \mathrm{e}_{j}\circ \alpha^{-1})\big){T'}^{j}T^{i} \\
	&=&({}^{\alpha}\!\mathrm{e}\widehat{\otimes}\mathrm{id}_{\mathbb{Z}[T]})\circ {}^{\alpha}\!\mathrm{e},
\end{array}
\end{align*}
which shows that ${}^{\alpha}\!\mathrm{e}$ is a restricted exponential homomorphism. 
\end{proof}

\begin{prop} \label{lem:restricted-eval1-auto} Let $\mathcal{B}$ be a complete topological ring and let $\mathrm{e}\colon \mathcal{B}\rightarrow \mathcal{B}\{T\}$ be a restricted exponential homomorphism. Then the compositions $$\varphi=\pi_{(1)}\circ \mathrm{e}\colon  \mathcal{B} \stackrel{\mathrm{e}}{\rightarrow} \mathcal{B}\{T\} \stackrel{T\mapsto 1}{\longrightarrow} \mathcal{B} \quad \textrm{and} \quad \psi=\pi_{(-1)}\circ \mathrm{e}\colon  \mathcal{B} \stackrel{\mathrm{e}}{\rightarrow} \mathcal{B}\{T\} \stackrel{T\mapsto -1}{\longrightarrow} \mathcal{B}$$ are  topological ring automorphisms of $\mathcal{B}$ inverse to each other.
\end{prop}
\begin{proof}
The homomorphisms $\varphi$ and $\psi$ are clearly continuous. 
Note that by definition, $\psi=\pi_{(1)}\circ\mathrm{e}_{\lambda(-1)}$. 
Furthermore, letting $\pi_{(1,1)}=\pi_{(1)}\circ\Delta\colon \mathcal{B}\{T,T'\}\rightarrow\mathcal{B}$ be the unique continuous $\mathcal{B}$-algebra homomorphism that maps $T$ and $T'$ to $1$, we have $$\psi\circ\varphi=\pi_{(1,1)}\circ(\mathrm{e}_{\lambda(-1)}\widehat{\otimes}\mathrm{id}_{\mathbb{Z}[T]})\circ \mathrm{e}\quad\textrm{and}\quad\varphi\circ\psi=\pi_{(1,1)}\circ(\mathrm{e}\widehat{\otimes}\mathrm{id}_{\mathbb{Z}[T']})\circ\mathrm{e}_{\lambda(-1)}.$$
Let $f\colon \mathcal{B}\rightarrow\mathcal{B}\{T,T'\}$ be the composition of $\mathrm{e}$ with the unique continuous $\mathcal{B}$-algebra homomorphism $\mathcal{B}\{T\}\rightarrow\mathcal{B}\{T,T'\}$ that maps $T$ to $T-T'$. Since $\mathrm{e}$ is a restricted exponential homomorphism, Equation \eqref{eq:rest-exp-condition} implies that $$(\mathrm{e}_{\lambda(-1)}\widehat{\otimes}\mathrm{id}_{\mathbb{Z}[T]})\circ \mathrm{e}=f=(\mathrm{e}\widehat{\otimes}\mathrm{id}_{\mathbb{Z}[T']})\circ\mathrm{e}_{\lambda(-1)}.$$ 
Since $\pi_{(1,1)}\circ f=\mathrm{id}_{\mathcal{B}}$, the assertion follows.
\end{proof}

\subsection{Topologically integrable derivations setting}\label{subsec:top-der-summary}
For the convenience of the reader, we now briefly translate, when the base topological ring $\mathcal{A}$ contains $\mathbb{Q}$, some of the main basic properties of restricted exponential homomorphisms established in subsection \ref{sub:exp-basic} in the language of topologically integrable derivations 
We first observe that the ring of invariants $\mathcal{B}^{\mathrm{e}}$ of the restricted exponential homomorphism associated to topologically integrable $\mathcal{A}$-derivation $\partial$ of $\mathcal{B}$ is equal to the kernel of $\partial$.

\begin{prop} Let $\mathcal{A}$ be a topological ring containing $\mathbb{Q}$, let $\mathcal{B}$ be a complete topological $\mathcal{A}$-algebra  and let $\partial$ be a topologically integrable $\mathcal{A}$-derivation of $\mathcal{B}$. Then the following hold :

a) For every element $f\in \operatorname{Ker}(\partial)$, the $\mathcal{A}$-derivation $f\partial$ of $\mathcal{B}$ is topologically integrable.

b) For every multiplicatively closed subset $S$ of $\operatorname{Ker}(\partial)$, the $\mathcal{A}$-derivation $\widehat{S^{-1}\partial}$ of $\widehat{S^{-1}\mathcal{B}}$ is topologically integrable.

c) For every surjective open homomorphism of complete topological rings $\pi\colon \mathcal{B}\rightarrow \mathcal{C}$ such that $\partial(\operatorname{Ker}\pi)\subset \operatorname{Ker}\pi$, the induced $\mathcal{A}$-derivation $\overline{\partial}$ of $\mathcal{C}\cong \mathcal{B}/\operatorname{Ker}\pi$ is topologically integrable.

d) For every topologically integrable $\mathcal{A}$-derivation $\partial'$ of $\mathcal{B}$ such $\partial\circ \partial'=\partial'\circ \partial$, the $\mathcal{A}$-derivation $\partial''=\partial+\partial'$ of $\mathcal{B}$ is topologically integrable. 
\end{prop}

\begin{proof} The assertions follow respectively from Propositions \ref{lem:replicat}, \ref{prop:Localization-Kernel}, \ref{prop:induced-quotient-exp} and \ref{lem:sum-restr-exp}
\end{proof}

\begin{prop} \label{prop:proj-lim-Der}Let $\mathcal{A}$ be a topological ring containing $\mathbb{Q}$ and let $(\mathcal{B}_n)_{n\in \mathbb{N}}$ be a countable inverse system of complete topological $\mathcal{A}$-algebras with continuous and surjective transition homomorphisms $p_{m,n}\colon \mathcal{B}_m\rightarrow \mathcal{B}_n$ for every $m\geq n\geq 0$. Let $\mathcal{B}$ be its limit and let $p_n\colon \mathcal{B}\rightarrow \mathcal{B}_n$, $n\in \mathbb{N}$, be the canonical continuous projections. Let $\partial_n\colon \mathcal{B}_n\rightarrow \mathcal{B}_n$, $n\in \mathbb{N}$, be a sequence of topologically integrable $\mathcal{A}$-derivations such that $\partial_n\circ p_{m,n}=p_{m,n}\circ \partial_m$ for all $m\geq n\geq 0$. Then there exists a unique topologically integrable $\mathcal{A}$-derivation $\partial=\varprojlim_{n\in \mathbb{N}} \partial_n$ of $\mathcal{B}$ such that  $\partial_n\circ p_n=p_n \circ \partial$ for every $n\in \mathbb{N}$.
\end{prop}

\begin{proof}
 The result is a consequence of Proposition \ref{lem:proj-lim-restricted} in the Appendix.
\end{proof}

\section{Sliced restricted exponential homomorphisms}\label{sec:4}
In this section, we provide structure results for exponential homomorphisms admitting local slices in a form which generalizes the classical description  \cite[Section 1.5]{M78} for discrete topological rings. 
\begin{defn}
Let $\mathcal{B}$ be a complete topological ring and let $\mathrm{e}\colon \mathcal{B}\rightarrow\mathcal{B}\{T\}$ be a restricted
exponential homomorphism. A \emph{local slice} for $\mathrm{e}$ is an element $s\in\mathcal{B}$
such that $\mathrm{e}(s)\in\mathcal{B}\{T\}$ is a non-constant polynomial of minimal possible degree. A \emph{slice} is a local slice $s$ such that  $\mathrm{e}(s)$ is a monic polynomial. 
\end{defn}
By definition, a restricted exponential homomorphism $\mathrm{e}:\mathcal{B}\to \mathcal{B}\{T\}$ of a complete topological ring $\mathcal{B}$ admits a local slice if and only if the image of $\mathrm{e}$ contains a non-constant polynomial. This holds in particular whenever the topology on $\mathcal{B}$ is the discrete one and $\mathrm{e}$ is not equal to the trivial exponential homomorphism $i_0$. For instance, for $\mathcal{B}=\mathbb{C}[u]$ endowed with the discrete topology, the restricted exponential homomorphism $\mathrm{e}\colon \mathbb{C}[u]\rightarrow \mathbb{C}[u][T]$ defined by $P(u)\mapsto P(u+T)$ has the element $s=u$ as a slice.  In contrast, the next example illustrates that when the topology on  $\mathcal{B}$ is not the discrete one, local slices do not exist in general.

\begin{example} \label{ex:no-slice-powerseries}   Letting $\mathcal{B}=\mathbb{C}[[u]]$ endowed with the $u$-adic topology, the restricted exponential homomorphism $$\mathrm{e}\colon \mathbb{C}[[u]]\rightarrow \mathbb{C}[[u]]\{T\}, \quad  u\mapsto u\sum_{i \in \mathbb{N}}(uT)^{i}=\sum_{i \in \mathbb{N}}u^{i+1}T^{i}$$ of Example \ref{ex:non-fc-kernel} does not admit a local slice. Indeed, if such a local slice existed then it would be an element of $\operatorname{Ker}\partial^2\setminus \operatorname{Ker}\partial$, where $\partial$ is the topologically integrable $\mathbb{C}$-derivation $u^2\tfrac{\partial}{\partial u}$ of $\mathbb{C}[[u]]$ 
corresponding to $\mathrm{e}$. But here, one has $\operatorname{Ker}\partial^n=\operatorname{Ker}\partial=\mathbb{C}$ for every $n\geq 1$, so that, in particular, $ \operatorname{Ker}\partial^2\setminus \operatorname{Ker}\partial=\emptyset$.
\end{example}
\begin{lem}\label{lem:leading-term-invariant-regular} Let $\mathcal{B}$ be a complete topological ring, let $\mathrm{e}\colon \mathcal{B}\rightarrow \mathcal{B}\{T\}$ be a restricted exponential homomorphism and let $\sigma$ be an element of $\mathcal{B}$ such that $\mathrm{e}(\sigma)\in \mathcal{B}\{T\}$ is a non-constant polynomial. Then the following hold:

(a) The leading coefficient of $\mathrm{e}(\sigma)$ is an element of $\mathcal{B}^{\mathrm{e}}$.

(b) If the leading coefficient $\mathrm{e}(\sigma)$ is invertible in $\mathcal{B}$ then $\sigma$ is a regular element of $\mathcal{B}$.
\end{lem}

\begin{proof}  
By applying $\mathrm{e}\widehat{\otimes}\mathrm{id}_{\mathbb{Z}[T]}$ to the identity $\mathrm{e}(\sigma)=\sum_{i=0}^m \mathrm{e}_i(\sigma)T^i$, we obtain the equality
\begin{align*}
\sum_{i=0}^{m}\mathrm{e}_i(\sigma)(T+T')^i  = ((\mathrm{id}_{\mathcal{B}}\widehat{\otimes} m)\circ \mathrm{e})(\sigma) =  ((\mathrm{e}\widehat{\otimes}\mathrm{id}_{\mathbb{Z}[T]})\circ \mathrm{e})(\sigma)  =  \sum_{i=0}^m \mathrm{e}(\mathrm{e}_i(\sigma)) T^i 
\end{align*}
in $\mathcal{B}\{T,T'\}$. By comparing the terms in $T^m$, we infer that $\mathrm{e}(\mathrm{e}_m(\sigma))=\mathrm{e}_m(\sigma)$. So, $\mathrm e_m(\sigma)\in \mathcal B^{\mathrm e}$, which proves (a).

Now assume that $\mathrm{e}_m(\sigma)$ is invertible in $\mathcal{B}$ and let  $b\in \mathcal{B}$ be an element such that $\sigma b=0$. Then we have $$0=\mathrm{e}(\sigma)\mathrm{e}(b)=(\sum_{i=0}^m\mathrm{e}_i(\sigma)T^i)(\mathrm{e}(b))=\mathrm{e}_m(\sigma)\mathrm{e}(b)T^m +\sum_{i=0}^{m-1} \mathrm{e}_i(\sigma)\mathrm{e}(b)T^i $$
from which we infer that for every $j\geq 0$, $\mathrm{e}_j(b)=-(\mathrm{e}_m(\sigma))^{-1}(\sum_{i=0}^{m-1}\mathrm{e}_i(\sigma)\mathrm{e}_{m+j-i}(b))$, hence, by induction, that for every integer $\ell >0$, $b=\mathrm{e}_0(b)$ belongs to the ideal of $\mathcal{B}$ generated by the elements $\mathrm{e}_n(b)$, $n\geq \ell$. Since $\mathrm{e}(b)\in \mathcal{B}\{T\}$, the sequence $(\mathrm{e}_n(b))_{n\in\mathbb{N}}$ converges to $0$ in $\mathcal{B}$, which implies, since $\mathcal{B}$ is complete, that $b=0$, hence that $\sigma$ is a regular element of $\mathcal{B}$. 
\end{proof}

\subsection{Restricted exponential homomorphisms with a slice}
We now describe the structure of complete topological rings $\mathcal{B}$ with a restricted exponential homomorphism $\mathrm{e}:\mathcal{B}\to \mathcal{B}\{T\}$ admitting a slice.

\begin{thm}\label{Theorem: Slice} 
\label{prop:Struct-Slice}\label{cor:slice} Let $\mathcal{B}$ be a complete topological ring, let $\mathrm{e}\colon \mathcal{B}\rightarrow\mathcal{B}\{T\}$ be a restricted
exponential homomorphism and let $s$ be a slice for $\mathrm{e}$. Then the homomorphism of topological $\mathcal{B}^{\mathrm{e}}$-algebras $$\theta:\mathcal{B}^{\mathrm{e}}\{X\}\to \mathcal{B},\quad X\mapsto s$$ 
is an isomorphism. Moreover, letting $\mathrm{e}_0:\mathcal{B}^{\mathrm{e}}\{X\}\to \mathcal{B}^{\mathrm{e}}\{X\}\{T\}$ be the exponential homomorphism of topological $\mathcal{B}^{\mathrm{e}}$-algebras defined by $X\mapsto (\theta^{-1} \hat{\otimes} \mathrm{id}_{\mathbb{Z}[T]})(\mathrm{e}(s))$,  
the following diagram of topological $\mathcal{B}^{\mathrm{e}}$-algebras is commutative \[\xymatrix{ \mathcal{B}^{\mathrm{e}}\{X\} \ar[r]^{\theta} \ar[d]_{\mathrm{e}_0} & \mathcal{B} \ar[d]^{\mathrm{e}} \\ \mathcal{B}^{\mathrm{e}}\{X\}\{T\} \ar[r]^-{\theta \hat{\otimes} \mathrm{id}} & \mathcal{B}\{T\}}.\]
\end{thm}

The rest of this subsection is devoted to the proof of Theorem \ref{prop:Struct-Slice}. We begin with the following lemma:

\begin{lem} \label{lem:Weier-div} Let $\mathcal{B}$ be a complete topological ring, let $\mathrm{e}\colon \mathcal{B}\rightarrow\mathcal{B}\{T\}$ be a restricted
exponential homomorphism and let $s$ be a slice for $\mathrm{e}$. Then for every $b\in \mathcal{B}$, there exists a unique pair consisting of an element $q\in \mathcal{B}\{T\}$ and a polynomial $r\in \mathcal{B}[T]$ of degree strictly less than $\deg_T(s)$ such that $\mathrm{e(b)}=q\mathrm{e}(s)+r$. Moreover, $r$ is a constant polynomial $r_0\in \mathcal{B}^{\mathrm{e}}$ and $q=\mathrm{e}(q(0))$. 
\end{lem}
\begin{proof}
 Write $\mathrm{e}(s)=\sum_{i=0}^{m} \mathrm{e}_i(s) T^i$, with $\mathrm{e}_m(s)=1$. Since $\mathrm{e}(s)\in \mathcal{B}\{T\}$ is a non-constant and monic, it follows from Weierstrass division theorem in $\mathcal{B}\{T\}$ (see e.g. \cite{Sal64}) that for every element $b\in \mathcal{B}$, there exist a unique element $q(T)\in \mathcal{B}\{T\}$ and a unique polynomial $r(T)\in \mathcal{B}\{T\}$ of degree $<m$ such that $\mathrm{e}(b)=q(T)\mathrm{e}(s)+r(T)$. Applying $\mathrm{e}\widehat{\otimes}\mathrm{id}_{\mathbb{Z}[T]}$ to the equality  $\mathrm{e}(b)=q(T)\mathrm{e}(s)+r(T)$, we get the identity 
\begin{align*}
(\mathrm{e}\widehat{\otimes}\mathrm{id}_{\mathbb{Z}[T]})(\mathrm{e}(b)) & = & q(T+T') (\mathrm{e}\widehat{\otimes}\mathrm{id}_{\mathbb{Z}[T]})(\mathrm{e}(s))+r(T+T') \\
& = &  (\mathrm{e}\widehat{\otimes}\mathrm{id}_{\mathbb{Z}[T]})(\mathrm{e}(s))(\mathrm{e}\widehat{\otimes}\mathrm{id}_{\mathbb{Z}[T]})(q(T))+(\mathrm{e}\widehat{\otimes}\mathrm{id}_{\mathbb{Z}[T]})(r(T))
\end{align*}
in $\mathcal{B}\{T,T'\}$ from which we infer the equality
$$(\mathrm{e}\widehat{\otimes}\mathrm{id}_{\mathbb{Z}[T]})(\mathrm{e}(s))\left(q(T+T')- (\mathrm{e}\widehat{\otimes}\mathrm{id}_{\mathbb{Z}[T]})(q(T))\right)=
(\mathrm{e}\widehat{\otimes}\mathrm{id}_{\mathbb{Z}[T]})(r(T))-r(T+T').$$
Since $(\mathrm{e}\widehat{\otimes}\mathrm{id}_{\mathbb{Z}[T]})(\mathrm{e}(s))$ is a monic polynomial of degree $m$ in $T$ over $\mathcal{B}\{T'\}$ whereas, on the other hand,  $r(T+T')-(\mathrm{e}\widehat{\otimes}\mathrm{id}_{\mathbb{Z}[T]})(r(T))$ is a polynomial of degree $<m$ in $T$ over $\mathcal{B}\{T'\}$, it follows that $q(T+T')=(\mathrm{e}\widehat{\otimes}\mathrm{id}_{\mathbb{Z}[T]})(q(T))$ and $r(T+T')=(\mathrm{e}\widehat{\otimes}\mathrm{id}_{\mathbb{Z}[T]})(r(T))$. By evaluating at $T=0$, we deduce in turn that 
 $q(T')=\mathrm{e}(q(0))$ and $r(T')=\mathrm{e}(r(0))$. Since $s$ is a slice and $r(T)$ is a polynomial of degree strictly less than that of $\mathrm{e}(s)$, the only possibility is that $r(T)$ is a constant polynomial, hence equal to $r(0)$, and that $r(0)=\mathrm{e}(r(0))$ is an element of $\mathcal{B}^{\mathrm{e}}$. 
\end{proof}

\begin{lem}\label{lem:Reynolds}With the assumption and notation of Lemma \ref{lem:Weier-div} above, the map $$R_s:\mathcal{B}\to \mathcal{B},\; b\mapsto r_0$$
is an idempotent endomorphism of topological $\mathcal{B}^{\mathrm{e}}$-algebras with image $\mathcal{B}^{\mathrm{e}}$. Moreover, the identity $R_s(R_s(b)b')=R_s(b)R_s(b')$ holds for every $b,b'\in \mathcal{B}$.
\end{lem}

\begin{proof} 
It is straightforward to verify that $R_s$ is a continuous endomorphism the underlying abelian topological group of $\mathcal{B}$. The fact that $R_s$ is homomorphism of rings follows on the other hand from the observation that for every $b,b'\in \mathcal{B}$, we have $$\mathrm{e}(b'b)=\mathrm{e}(b')\mathrm{e}(b)=(q'\mathrm{e}(s)+r_0')(q\mathrm{e}(s)+r_0)=q''\mathrm{e}(s)+r_0'r_0$$ 
which implies, by the uniqueness part of Lemma \ref{lem:Weier-div}, that 
$R_s(b'b)=r_0'r_0=R_s(b')R_s(b)$. 
The same argument also implies that $R_s$ is a homomorphism of $\mathcal{B}^{\mathrm{e}}$-algebras.  Since, by construction, the image of $R_s$ is contained in $\mathcal{B}^{\mathrm{e}}$,
it follows that $R_s$ is idempotent and that $$\mathrm{e}(R_s(b)b')=R_s(b)\mathrm{e}(b')=R_s(b)(q'\mathrm{e}(s)+R_s(b')),$$
 from which we infer, by the uniqueness part of Lemma \ref{lem:Weier-div} again that $R_s(R_s(b)b')=R_s(b)R_s(b')$.
\end{proof}

\begin{defn} \label{def:Dixmier} With the notation of Lemma \ref{lem:Weier-div} above, by analogy with \cite[1.1.9]{F06}, we call the  homomorphism  of topological rings $R_s:\mathcal{B}\to \mathcal{B}$ the \emph{Dixmier-Reynolds homomomorphism} associated to  $s$.
\end{defn}

\begin{proof}[{Proof of Theorem \ref{prop:Struct-Slice}}] 
 A repeated application of Lemma \ref{lem:Weier-div} gives by induction the existence of a unique sequence $(r_n)_{n\in \mathbb{N}}$ of elements of $\mathcal{B}^{e}$  such that $\mathrm{e}(b)=\sum_{n\in \mathbb{N}}r_n (\mathrm{e}(s))^n$. Since $\mathrm{e}(b)\in \mathcal{B}\{T\}$, the sequence of coefficients $(\mathrm{e}_i(b))_{i\in \mathbb{N}}$ of $\mathrm{e}(b)$ converges to $0$ with respect to the topology on $\mathcal{B}$. Since $\mathrm{e}(s)$ is monic polynomial, this holds if and only if the sequence  $(r_n)_{n\in \mathbb{N}}$ converges to $0$ in $\mathcal{B}$, hence in $\mathcal{B}^{\mathrm{e}}$ since $\mathcal{B}^{\mathrm{e}}$ is a complete subring of $\mathcal{B}$ by Proposition \ref{lem:kernel} (a). It follows that $\sum_{n\in \mathbb{N}} r_n X^n$ is a well-defined element of $\mathcal{B}^{\mathrm{e}}\{X\}$ whose image by $\theta$ equals $b=\mathrm{e}(b)(0)=\sum_{n\in \mathbb{N}}r_ns^n$, showing the surjectivity of $\theta$. 
 
 Suppose that $\theta$ is not injective and let $a=\sum_{i\in \mathbb{N}}a_{i}X^{i}\in\mathcal{B}^{\mathrm{e}}\{X\}$ be a nonzero element of minimal order $\mathrm{ord}_{0}(a)=\min\{i,\;a_{i}\neq0\}$ such that $\theta(a)=\sum_{i\in \mathbb{N}}a_{i}s^{i}=0$.  Since by Lemma \ref{lem:leading-term-invariant-regular} (b), $s$ is a regular element of $\mathcal{B}$, the minimality of $\mathrm{ord}_{0}(a)$
implies that $a_{0}\neq0$. On the other hand, since $R_s(s)=0$, 
we have 
\[
a_{0}=\sum_{i\in \mathbb{N}}a_{i}R_s(s)^{i}=R_s(\theta(a))=0,
\]
a contradiction. This shows that $\theta$ is injective, hence an isomorphism. The second assertion is a straighforward consequence of the construction. 
\end{proof}

\subsection{Restricted exponential homomorphism with a local slice}
We now consider a restricted exponential homomorphism $\mathrm{e}:\mathcal{B}\to \mathcal{B}\{T\}$ with a local slice $s\in \mathcal{B}$ which is not a slice. Writing $\mathrm{e}(s)=\sum_{i=0}^{m} \mathrm{e}_i(s)T^i$, $s_m:=\mathrm{e}_m(s)$ is a non-zero element of $\mathcal{B}^{\mathrm{e}}\setminus \{1\}$ by Lemma \ref{lem:leading-term-invariant-regular} (a).  
Let  $$\widetilde{j}_{s_{m}}\colon \mathcal{B}\rightarrow\widehat{\mathcal{B}_{s_{m}}}=\widehat{S^{-1} \mathcal{B}} \quad \textrm{and} \quad \bar{j}_{s_{m}}\colon \mathcal{B}^{\mathrm{e}}\rightarrow\widehat{\mathcal{B}_{s_{m}}^{\mathrm{e}}}=\widehat{S^{-1}\mathcal{B}^{\mathrm{e}}}$$
be the separated completed localization homomorphisms of $\mathcal{B}$ and
$\mathcal{B}^{\mathrm{e}}$ with respect to the multiplicative subset $S=\{s_{m}^{n}\}_{n \in \mathbb{N}}$ of $\mathcal{B}^{\mathrm{e}}\subseteq\mathcal{B}$.
Since $\mathcal{B}^{\mathrm{e}}$ is closed in $\mathcal{B}$ by Proposition \ref{lem:kernel}, the homomorphism of topological rings $\widehat{i}_{s_{m}}\colon \widehat{\mathcal{B}_{s_{m}}^{\mathrm{e}}}\rightarrow\widehat{\mathcal{B}_{s_{m}}}$ induced by the inclusion $i\colon \mathcal{B}^{\mathrm{e}}\hookrightarrow\mathcal{B}$
is injective by Lemma \ref{lem:Injectivity-Loc}. We henceforth consider
$\widehat{\mathcal{B}_{s_{m}}}$ as a $\widehat{\mathcal{B}_{s_{m}}^{\mathrm{e}}}$-algebra
via this homomorphism. By Proposition \ref{prop:Localization-Kernel},
there exists a unique restricted exponential homomorphism $\widehat{\mathrm{e}_{s_{m}}}:=\widehat{S^{-1}\mathrm{e}}\colon \widehat{\mathcal{B}_{s_{m}}}\rightarrow\widehat{\mathcal{B}_{s_{m}}}\{T\}$
such that the following diagram commutes
\[
\xymatrix{ \mathcal{B} \ar[r]^-{\mathrm{e}} \ar[d]_{\widetilde{j}_{s_{m}}} & \mathcal{B}\widehat{\otimes}_{\mathbb{Z}}\mathbb{Z}[T] \ar[d]^{\widetilde{j}_{s_{m}}\widehat{\otimes}\mathrm{id}} \\
\widehat{\mathcal{B}_{s_{m}}} \ar[r]^-{\widehat{\mathrm{e}_{s_{m}}}} & \widehat{\mathcal{B}_{s_{m}}}\widehat{\otimes}_{\mathbb{Z}}\mathbb{Z}[T].}
\] 

Now assume that $\widehat{\mathcal{B}_{s_{m}}}$ is not the zero ring and let $\sigma$ be the image of the element $s_{m}^{-1}s\in \mathcal{B}_{s_m}$ in $\widehat{\mathcal{B}_{s_{m}}}$ by the separated completion homomorphism $\mathcal{B}_{s_m}\to \widehat{\mathcal{B}_{s_m}}$. Since $\widehat{\mathcal{B}_{s_m}}\neq \{0\}$, the image $\tilde{j}_{s_m}(s_m)$ of $s_{m}$ in $\widehat{\mathcal{B}_{s_{m}}}$ is a nonzero invertible element (see Corollary \ref{cor:criterion-zero-loc}) and, by definition of $\widehat{\mathrm{e}_{s_m}}$, $$\widehat{\mathrm{e}_{s_{m}}}(\sigma)=\sum_{i=0}^m (\tilde{j}_{s_m}(s_m))^{-1}\tilde{j}_{s_m}(\mathrm{e}_i(s))T^i$$ 
is a monic polynomial of degree $m$. Theorem \ref{prop:Struct-Slice} then implies the following:

\begin{cor}     \label{lem:localized-sliced} With the notation above,  assume that $\widehat{\mathcal{B}_{s_{m}}}$ is not the zero ring and that $\sigma$ \emph{is a slice} for the restricted exponential homomorphism $\widehat{\mathrm{e}_{s_{m}}}\colon \widehat{\mathcal{B}_{s_{m}}}\rightarrow\widehat{\mathcal{B}_{s_{m}}}\{T\}$.  Then the homomorphism of topological $\widehat{\mathcal{B}^{\mathrm{e}}_{s_{m}}}$-algebras $\theta_{\sigma}:\widehat{\mathcal{B}^{\mathrm{e}}_{s_m}}\{X\}\to \widehat{\mathcal{B}_{s_{m}}}$, $\quad X\mapsto \sigma$ 
is an isomorphism.
\end{cor}  

We do not know in general whether in Corollary \ref{lem:localized-sliced} the condition that $\sigma$ is a slice can be relaxed.
More precisely, we do not know how to exclude the possibility that after passing to the separated completion $\mathcal{B}_{s_m}\to \widehat{\mathcal{B}_{s_m}}$, $\sigma$ is not a local slice for  
$\widehat{\mathrm{e}_{s_{m}}}$, i.e. that there exists an element $\sigma'$ of $\widehat{\mathcal{B}_{s_m}}$ such that $\widehat{\mathrm{e}_{s_{m}}}(\sigma')$ is a polynomial of degree strictly less than  that of the monic polynomial $\widehat{\mathrm{e}_{s_{m}}}(\sigma)$. It is nevertheless  readily verified that $\sigma$ is automatically a slice when the topology on $\mathcal{B}$ is the discrete one or when $s$ is a local slice with $\mathrm{e}(s)$ of degree $1$. The latter property being automatically satisfied when the base topological ring $\mathcal{A}$ contains $\mathbb{Q}$. Indeed, through the correspondence between restricted exponential homomorphisms $\mathrm{e}:\mathcal{B}\to \mathcal{B}\{T\}$ on complete topological algebras $\mathcal{B}$ containing $\mathbb{Q}$ and topologically integrable derivations $\partial$ of $\mathcal{B}$, local slices for $\mathrm{e}$ correspond to elements of $\operatorname{Ker}\partial^2\setminus \operatorname{Ker}\partial$ and a local slice $s$ is a slice if and only if $\mathrm{e}(s)=s+T$. Corollary \ref{lem:localized-sliced} then implies the following:
\begin{prop} \label{prop:Loc-Struct-Slice-Der}  Let $\mathcal{A}$ be a topological ring containing $\mathbb{Q}$, let $\mathcal{B}$ be a complete topological $\mathcal{A}$-algebra  and let $\partial$ be a topologically integrable $\mathcal{A}$-derivation of $\mathcal{B}$. If $\operatorname{Ker}\partial^2\setminus \operatorname{Ker}\partial$ is not empty then for every $s\in \operatorname{Ker}\partial^2\setminus \operatorname{Ker}\partial$ with $\partial(s)=s_1\in \operatorname{Ker}\partial$, there exists an isomorphism of topological $\widehat{(\operatorname{Ker}\partial)_{s_1}}$-algebras $$\widehat{\mathcal{B}_{s_1}}\stackrel{\cong}{\rightarrow}\widehat{(\operatorname{Ker}\partial)_{s_1}}\{S\}$$ which maps the induced topologically integrable $\widehat{(\operatorname{Ker}\partial)_{s_1}}$-derivation $\widehat{\partial_{s_1}}$ of $\widehat{\mathcal{B}_{s_1}}$ onto the topologically integrable $\widehat{(\operatorname{Ker}\partial)_{s_1}}$-derivation $\tfrac{\partial}{\partial{S}}$ of $\widehat{(\operatorname{Ker}\partial)_{s_1}}\{S\}$. 
\end{prop}

\section{Examples: additive group actions on the affine ind-space}
\label{sec:5} We illustrate the results of the previous sections on examples of restricted exponential homomorphisms and associated topologically integrable derivations on pro-polynomial rings 
which correspond, under the anti-equivalence of categories 
reviewed in subsection \ref{sub:abstract-nonsense} to actions of the additive group $\mathbb{G}_a$ on the "affine ind-space" $\mathbb{A}^{\infty}$. In particular, we give a construction which associates to every $\mathbb{G}_{a,k}$-action on an affine scheme of finite type over a field $k$ of characteristic zero an action of $\mathbb{G}_{a,k}$ on $\mathbb{A}^{\infty}_k$. 

\medskip

Our construction builds on the following general observation: 

\begin{prop} \label{ex:CompleteTensorRep}Let $R$ be a finitely generated algebra over a field $k$. Then the covariant functor $$R\widehat{\otimes}_k -\colon (\mathrm{CRTop}_{/k})\to (\mathrm{Sets})$$ 
associating to a complete topological $k$-algebra $\mathcal{A}$ the completed tensor product $R\widehat{\otimes}_k \mathcal{A}$ is representable by a pro-polynomial ring.
\end{prop}
\begin{proof} 
Since $R$ is a finitely generated, $k$-algebra, it is a countable $k$-vector space. We can thus write $R=\varinjlim_{n \in \mathbb{N}} V_n=\bigcup_{n\in \mathbb{N}} V_n$ where the $V_n$ are an increasing sequence of finite dimensional $k$-vector spaces $V_0\subset \cdots \subset V_n \subset V_{n+1} \subset \cdots$ which form an exhaustion of $R$. For every $n\leq n'$, the inclusion $V_n \subset V_{n'}$ induces a surjection $ V_{n'}^{\vee} \to  V_{n}^{\vee}$ between the duals of $V_{n'}$ and $V_n$ respectively, hence a surjective $k$-algebra homomorphism $\mathrm{Sym}^{\cdot}(V_{n'}^{\vee}) \to \mathrm{Sym}^{\cdot}(V_{n}^{\vee})$ between the symmetric $k$-algebras of $V_{n'}^\vee$ and $V_n^\vee$, respectively. 
Viewing each $P_n=\mathrm{Sym}^{\cdot}(V_{n}^{\vee})$ as endowed with the discrete topology, the $k$-algebra 
$\mathcal{P}=\varprojlim_{n\in \mathbb{N}} P_n$ endowed with the limit topology is a pro-polynomial ring with a fundamental system of.
open neighborhoods of $0_{\mathcal{P}}$ consisting of the kernel $\mathfrak{p}_n$ of the canonical surjections $p_n:\mathcal{P}\to \mathrm{Sym}(V_n^\vee)$.

Now let  $\mathcal{A}$ be a  complete topological $k$-algebra an let $(\mathfrak{a}_m)_{m\in \mathbb{N}}$ be any fundamental system of open neighborhoods of $0_{\mathcal{A}}$.Since tensor product commutes with colimits and the $k$-vector spaces $V_n$ are finite dimensional, we have natural isomorphisms 
\begin{align*}
\begin{array}{rcl}
     R\widehat{\otimes}_k \mathcal{A} = \varprojlim_{m\in \mathbb{N}} (R\otimes_k A/\mathfrak{a}_m)  & = & \varprojlim_{m\in \mathbb{N}} ((\varinjlim_{n \in \mathbb{N}}V_n) \otimes_k A/\mathfrak{a}_m) \\ 
     & \cong &  \varprojlim_{m\in \mathbb{N}}( \varinjlim_{n \in \mathbb{N}}  (V_n \otimes_k A_m)) \\
     & \cong & \varprojlim_{m\in \mathbb{N}}( \varinjlim_{n \in \mathbb{N}}( \mathrm{Hom}_{k-\mathrm{mod}} (V_n^\vee, A_m))\\
     & \cong & \varprojlim_{m\in \mathbb{N}}( \varinjlim_{n \in \mathbb{N}}( \mathrm{Hom}_{k-\mathrm{alg}} (\mathrm{Sym}(V_n^\vee), A/\mathfrak{a}_m))\\
     & \cong & 
     \varprojlim_{m\in \mathbb{N}}( \varinjlim_{n \in \mathbb{N}}( \mathrm{Hom}_{k-\mathrm{alg}} (\mathcal{P}/\mathfrak{p}_n,A/\mathfrak{a}_m))\\
     & \cong & \mathrm{CHom}_k(\mathcal{P},\mathcal{A})
\end{array}
\end{align*}
The above isomorphism $ \Phi_{\mathcal{A}}:\mathrm{CHom}_k(\mathcal{P},\mathcal{A})  R\widehat{\otimes}_k \mathcal{A}$ 
is easily seen from the construction to be functorial in $\mathcal{A}$, defining an isomorphism of covariant functors $\Phi\colon \mathrm{CHom}_k(\mathcal{P},-) \to R\widehat{\otimes}_k -$ 
 which shows that the pro-polynomial ring $\mathcal{P}$ represents the functor $R\widehat{\otimes}_k -$. 
 
 The universal element $u=\Phi_{\mathcal{P}}(\mathrm{id}_{\mathcal{P}})\in R\widehat{\otimes}_k \mathcal{P}$ can be described as follows. For every $n\in \mathbb{N}$, let $u_n\in R\otimes_k P_n=R\otimes_k \mathrm{Sym}^{\cdot} (V_n^\vee)$ be the image by the natural homomorphism $$V_n\otimes_k V_n^{\vee}\rightarrow R\otimes_k \mathrm{Sym}^{\cdot} (V_n^\vee)$$
of the element corresponding to $\mathrm{id}_{V_n}$ via the isomorphism $\mathrm{Hom}_k(V_n,V_n)\cong V_n\otimes_k V_n^{\vee}$. The collection of elements $u_n\in R\otimes_k P_n$ is an inverse system with respect to the projection homomorphisms $R\otimes_k P_{n'}\rightarrow R\otimes_k P_n$, $n \leq  n'$ and we have  $u=\varprojlim_{n\in \mathbb{N}} u_n \in \varprojlim_{n\in \mathbb{N}} R\otimes_k P_n=R\widehat{\otimes}_k \mathcal{P}$.
\end{proof}

 \label{ex:coordinate-ring} Through the anti-equivalence between the category $(\mathrm{CRTop}_{/k})$ and the category of restricted strict affine ind-$k$-schemes discussed in subsection \ref{sub:abstract-nonsense}, Proposition \ref{ex:CompleteTensorRep} has the following interpretation. Every affine $k$-scheme $X=\mathrm{Spec}(R)$  determines 
a covariant functor $$F_X=\mathrm{Mor}(X,\mathbb{A}^1_k):(\mathrm{Aff}_{/k})^{\mathrm{opp}}\to (\mathrm{Sets}),\; S \mapsto \mathrm{Hom}_{S}(X\times_{\mathrm{Spec}(k)}  S,\mathbb{A}^1_S)=R\otimes_k\Gamma(S,\mathcal{O}_S).$$
When $X$ is of finite type over $k$, the fact that as a $k$-vector space, $R=\varinjlim_{n\in \mathbb{N}} V_n$ for a chain $(V_0\subset V_1\subset \cdots \subset V_n\subset \cdots)_{n\in \mathbb{N}}$ of finite dimensional $k$-vector spaces implies that $F_X$ is a restricted strict affine ind-$k$-scheme isomorphic to the colimit\footnote{Taken in the category of set-valued contravariant functors on the category of affine $k$-schemes.}  $\varinjlim_{n\in \mathbb{N}} \mathrm{Spec}(\mathrm{Sym}(V_n^\vee))$. 

It follows in particular that 
the isomorphism type of $F_X=\mathrm{Mor}(X,\mathbb{A}^1_k)$ as a restricted strict affine ind-$k$-scheme is independant of the affine $k$-scheme of finite type $X$: it is isomorphic to the colimit $$\mathbb{A}^\infty_k:=\varinjlim_{n\in \mathbb{N}} \mathbb{A}^n_k$$  
for a chain of closed embeddings $(\mathbb{A}^0_k=\mathrm{Spec}(k)\hookrightarrow \mathbb{A}^1_k\hookrightarrow \cdots \hookrightarrow \mathbb{A}^n_k\hookrightarrow \cdots )_{n\in \mathbb{N}}$
as successive linear subspaces, that is, to the restricted strict affine ind-$k$-scheme corresponding to a pro-polynomial ring $\mathcal{P}$ in countably many indeterminates over $k$.

\medskip

Keeping the notation above, assume further that the affine $k$-scheme of finite $X$ is endowed with an action $\mu:\mathbb{G}_{a,k}\times_k X\to X$ of the additive group scheme $\mathbb{G}_{a,k}$. Then the natural transformation $\hat{\mu}:\mathbb{G}_{a,k}\times_k F_X\to F_X$ given by pre-composing morphisms from $X$ to $\mathbb{A}^1_k$ with the action $\mu$ is an action of $\mathbb{G}_{a,k}$ on the restricted strict affine ind-scheme $F_X$. Through the anti-equivalence above, the latter corresponds in turn to a certain exponential $k$-homomorphism $\mathrm{e}_{X,\mu}:\mathcal{P}\to \mathcal{P}\{T\}$ on a pro-polynomial ring $\mathcal{P}$. In the rest of this section, we assume further for simplicity that $k$ is a field of characteristic zero and proceed to give a more explicit construction and description of the topologically integrable derivation $\partial_{X,\mu}$ corresponding to this exponential homomorphism.

\medskip 

Under our hypothesis, the $\mathbb{G}_{a,k}$-action $\mu:\mathbb{G}_{a,k}\times_k X\to X$ on $X=\mathrm{Spec}(R)$ corresponds to a classical exponential $k$-homomorphism $\exp(T\delta):R\to R[T]$ associated to a locally nilpotent $k$-derivation $\delta$ of $R$. Then $R$ admits an exhaustion by a countable family $\mathcal{W}=\{W_n\}_{n\in \mathbb{N}}$ of finite dimensional $\delta$-stable $k$-vector subspaces. Indeed, given any exhaustion of $R$ by a countable family $\mathcal{V}=\{V_n\}_{n\in \mathbb{N}}$ of finite dimensional $k$-vector subspaces, the fact that $\delta$ is locally nilpotent implies that for every $n\in \mathbb{N}$, the $k$-vector subspace $W_n$ generated by the elements $\delta^m(f)$, $m\geq 0$, where the elements $f$ run through a $k$-basis of $V_n$ is finite dimensional and $\delta$-stable. Furthermore, since $V_n\subseteq V_m$ for every $m\geq n$ and $V_n\subseteq W_n$, we have $V_n\subseteq W_n\subseteq W_m$ so that the $W_m$ form an increasing exhaustion of $R$ by $\delta$-stable finite dimensional $k$-vector subspaces. 

Let $\mathcal{W}=\{W_m\}$ be a $\delta$-stable exhaustion of $R$ as above. Then, for every $m\in \mathbb{N}$, the restriction of $\delta$ to $W_m$ is a nilpotent $k$-linear endomorphism $\delta_m$ of $W_m$. The dual endomorphism $\delta_m^{\vee}$ of $W_m^{\vee}$ defines a unique $k$-derivation $\partial_m$ of the symmetric algebra $\mathrm{Sym}^{\cdot}(W_m^{\vee})$ of $W_m^{\vee}$, which is locally nilpotent.  The collection of so-defined locally nilpotent $k$-derivations $\partial_m$ of the $k$-algebras $\mathrm{Sym}^{\cdot}(W_m^{\vee})$ form an inverse system with respect to the surjective projection homomorphisms $p_{m,n}\colon \mathrm{Sym}^{\cdot}(W_m^{\vee})\rightarrow \mathrm{Sym}^{\cdot}(W_n^{\vee})$ associated to the inclusions $W_n\subseteq W_m$, $m\geq n$. By Proposition \ref{prop:Loc-Struct-Slice-Der}, there exists a unique topologically integrable $k$-derivation $\partial_{X,\mu}$ of the pro-polynomial $k$-algebra  $k$-algebra $\mathcal{P}=\varprojlim_{n\in \mathbb{N}} \mathrm{Sym}^{\cdot}(W_n^{\vee})$ representing the covariant functor $R\widehat{\otimes}_k-$ of Proposition \ref{ex:CompleteTensorRep} such that for every $n\in \mathbb{N}$, we have $\partial_n\circ p_n=p_n\circ \partial_{X,\mu}$, where $p_n\colon\mathcal{P}\rightarrow  \mathrm{Sym}^{\cdot}(W_n^{\vee})$ is the canonical continuous projection.

\begin{example}\label{exa:translation}
As a concrete illustration of the construction above, consider  the locally nilpotent $k$-derivation $\delta=\partial/\partial x$ of $R=k[x]$ corresponding to the action $\mu$ of  $\mathbb{G}_{a,k}$ on $\mathbb{A}^1_k$ by translations and the exhaustion of $R$ by the $\delta$-stable subspaces $$W_n=k[x]_{\leq n}=k\langle x^0,\ldots, x^n\rangle, \quad n\in \mathbb{N},$$ consisting of polynomials of degree less than or equal to $n$. For every $n\in \mathbb{N}$, the algebra   $\mathrm{Sym}^{\cdot}(W_n^{\vee})$ is isomorphic to the polynomial ring $k[X_0,\ldots, X_n]$, where $(X_0,\ldots, X_n,\ldots )$ is the family of elements of the dual $R^{\vee}$ of $R$ as a $k$-vector space defined by $X_i(x_j)=\delta_{i,j}$ for every $i,j\in \mathbb{N}$. Putting $P_n=\mathrm{Sym}^{\cdot}(W_n^{\vee})=k[X_0,\ldots, X_n]$ and letting, for every $m\geq n$ 
$p_{n,m}:P_m\to P_n$ the surjection with kernel $(X_{n+1},\ldots, X_m)P_m$, the pro-polynomial $k$-algebra $\mathcal{P}=
\varprojlim_{n\in \mathbb{N}} P_n$ is isomorphic to the separated completion of the polynomial ring $k[(X_i)_{i\in \mathbb{N}}]$ with respect to the topology induced by the fundamental system of open ideals $\mathfrak{a}_n=(X_i)_{i\geq n}k[(X_i)_{i\in \mathbb{N}}]$. 
On the other hand, the $k$-derivation $\partial_n$ of  $P_n$ 
is the composition of the $k$-derivation $$\tilde{\partial}_n=\sum_{i=0}^n (i+1)X_{i+1}\frac{\partial}{\partial X_i}:P_n\to P_{n+1}$$ with the projection $p_{n+1,n}:P_{n+1}\to P_n$. It is triangular, hence locally nilpotent. The corresponding topologically integrable $k$-derivation $\partial_{\mathbb{A}^1_k,\mu}=\varprojlim_{n\in \mathbb{N}} \partial_n$ of $\mathcal{P}$ coincides with the topologically integrable $k$-derivation $\widehat{\partial}$ of $\mathcal{P}$ induced through Lemma \ref{lem:Pro-LND-CompExt} by the uniformly equicontinuous derivation of $k[(X_i)_{i\in \mathbb{N}}]$ defined by $\partial(X_i)=(i+1)X_{i+1}$.

It is an instance of a topologically integrable $k$-derivation with trivial kernel $\mathrm{Ker}(\partial_{\mathbb{A}^1_k,\mu})=k$, in fact, one has more generally $\mathrm{Ker}(\partial_{\mathbb{A}^1_k,\mu}^n)=k$ for every $n\geq 1$. Indeed, it is straightforward to check by induction on $n$ that the kernel of the $k$-derivation $\tilde{\partial}_n$ equals $k$. Since $p_{n+1}\circ \partial_{\mathbb{A}^1_k,\mu}=\tilde{\partial}_n\circ p_n$, where $p_n:\mathcal{P}\to P_n$ is the canonical surjection, it follows that $\mathrm{Ker}(\partial_{\mathbb{A}^1_k,\mu})=k$. 
Now if
$\operatorname{Ker}(\partial_{\mathbb{A}^1_k,\mu}^n)\setminus
\operatorname{Ker}(\partial_{\mathbb{A}^1_k,\mu}^{n-1})\neq \emptyset$ for some
$n\geq 2$, then there would exist an element
$s\in \operatorname{Ker}(\partial_{\mathbb{A}^1_k,\mu}^2)\setminus
\operatorname{Ker}(\partial_{\mathbb{A}^1_k,\mu})$. Letting $s_1=\partial_{\mathbb{A}^1_k,\mu}(s)\in k\setminus\{0\}$,
it would follows from Proposition \ref{prop:Loc-Struct-Slice-Der} that
$\mathcal{P}=\widehat{\mathcal{P}_{s_1}}\cong \widehat{k_{s_1}}\{S\}=k[S]$ is a polynomial ring in one variable $S$ over $k$ endowed with the discrete topology, which is
impossible. Thus
$\operatorname{Ker}(\partial_{\mathbb{A}^1_k,\mu}^n)=\operatorname{Ker}(\partial_{\mathbb{A}^1_k,\mu})=k$ for every
$n\geq 1$.
\end{example}

\appendix
\section*{Appendix on topological algebra}
\renewcommand{\thethm}{\Alph{subsection}.\arabic{thm}}
\renewcommand{\thesubsection}{\Alph{subsection}}
In this appendix we  collect general results on topological groups, rings and modules. Standard references for theses topics are for instance Bourbaki \cite[Chapter~III]{BourbakiGT}, \cite[Chapter~III]{BourbakiCA} and Northcott \cite{N68}. 

\subsection{Topological groups} 

\begin{defn}
A topological abelian group is an abelian group $G$
endowed with a topology for which the map $G\times G\rightarrow G$,
$(x,y)\mapsto x-y$ is continuous. The topology on $G$ is called linear
if $G$ has a fundamental system of open neighborhoods of its neutral
element $0_G$ consisting of open subgroups.    
\end{defn}

We henceforth denote by $(\mathrm{GTop})$ the category whose objects are topological abelian groups $G$ endowed with a linear topology for which there exists a \emph{countable} fundamental system of open neighborhoods of $0_G$ and whose morphisms are continuous homomorphisms. 
In the sequel, unless otherwise stated, the term topological group will always refer to an object $G$ of $(\mathrm{GTop})$.  
The underlying topological space of such a $G$ is in particular a
first-countable topological space. Given a countable fundamental system of open neighborhoods
$(H_n)_{n \in \mathbb{N}}$ of $0_G$ parametrized by the set $\mathbb{N}$ of
non-negative integers, we will always assume in addition
that $H_{0}=G$ and that $H_{m}\subseteq H_n$ whenever $m\geq n$. A continuous homomorphism $f\colon G\rightarrow G'$ 
between objects of $(\mathrm{GTop})$  will be referred to as a \emph{homomorphism of topological groups}. Such homomorphisms are automatically uniformly continuous in the sense of \cite[II.2.1]{BourbakiGT}.

\subsubsection{Separated completion of a topological group}\label{sep:completions}

A topological group is called \emph{separated} if it is separated as a topological space. This holds if and only if the     intersection of all open subgroups of $G$ consists of the neutral element $0_G$ only, hence,  since every open subgroup of a topological group is also closed \cite[III.2.1 Corollary to Proposition 4]{BourbakiGT}, if and only if $\{0_G\}$ is a closed subset of $G$.

\begin{defn}
\label{def:Convergence} A family $(x_{i})_{i\in I}$ of elements of topological group $G$ indexed by a countable set $I$ is:

a) \emph{Cauchy} if for every open subgroup $H$ of $G$
 there exists a finite subset $J(H)$ of $I$ such that $x_{i}-x_{j}\in H$ for all $i,j\in I\setminus J(H)$.

b) \emph{Convergent} to an element $x\in G$ if for every open subgroup $H$ of $G$ 
there exists a finite subset $J(H)$ such that $x_{i}-x\in H$ for all $i\in I\setminus J(H)$.
\end{defn}

\begin{defn} A topological group $G$ is called \emph{complete} if it is separated and every Cauchy family $(x_{i})_{i\in I}$ of elements of $G$ is convergent in $G$.
 We denote by   $(\mathrm{CGTop})$ the full subcategory of  $(\mathrm{GTop})$ whose objects are complete topological groups.
\end{defn}

\begin{prop} 
\label{lem:Compl-Homo-Extension}
The inclusion functor $i:(\mathrm{CGTop})\to (\mathrm{GTop})$ has a left adjoint 
\begin{equation}\label{eq:sep_comp}
\widehat{\;}:(\mathrm{GTop})\to (\mathrm{CGTop}).
\end{equation}
\end{prop}
\begin{proof}
By \cite[III.2.6 Proposition 5]{BourbakiCA} and \cite[III.3.4 Proposition 8]{BourbakiGT}, for a topological group $G$, the functor 
\begin{equation}
F_G:=\mathrm{Hom}_{(\mathrm{GTop})}(G,-):(\mathrm{CGtop}) \to (\mathrm{Sets})    
\end{equation}
is representable by a complete topological group $\widehat{G}$, equal to the limit in the category of topological groups of the inductive system of discrete topological groups $G/H$, where $H$ ranges through the set $\Gamma$ of open subgroups of $G$. Note that since any countable fundamental system $(H_n)_{n \in \mathbb{N}}$ of open neighborhoods of $0_G$ is cofinal in $\Gamma$, the canonical homomorphism
  $ \widehat{G}\rightarrow \varprojlim_{n\in\mathbb{N}}G/H_{n}$ is an
  isomorphism of topological groups. It follows that $\widehat{G}$ 
has a countable fundamental system of
  open subgroups  consisting of the kernels of the canonical
  projections $\widehat{p}_{H_n}:\widehat{G}\to G/H_n$, $n\in \mathbb{N}$, hence that $\widehat{G}$ is indeed a topological group in our restricted sense. 
For every homomorphism of topological groups $h:G\to G'$, the universal properties of $\widehat{G}$ and $\widehat{G'}$ imply that the natural transformation $F_{G'}\to F_G$ induced by the composition with $h$ uniquely corresponds to a homomorphism of complete topological groups $\widehat{h}:\widehat{G}\to \widehat{G'}$. It is then straightforward to verify that association  
$$\widehat{\;}:(\mathrm{GTop})\to (\mathrm{CGTop}),\quad G\mapsto \widehat{G}, \quad (h:G\to G')\mapsto (\widehat{h}:\widehat{G}\to \widehat{G'}) $$
is a functor with the desired property. 
\end{proof}

\begin{defn}
\label{def:Separated-Completion}
For a topological group $G$, the complete topological group $\widehat{G}$ is called the \emph{separated completion} of $G$. The homomorphism of topological groups $c:G\to \widehat{G}$ determined by the unit of the adjunction is called the \emph{separated completion homomorphism}.    
\end{defn}
\begin{prop}[{\cite[III.7.3 Proposition 2]{BourbakiGT}}] \label{prop:sep-comp}
 For a topological group $G$, the image of separated completion homomorphism $c:G\to \widehat{G}$ is a dense subgroup of $\widehat{G}$. Moreover, the induced homomorphism of topological groups $c:G\to c(G)$ is open, with kernel equal to the closure in $G$ of the subset $\{0_G\}$. 
\end{prop} 

\begin{lem}\label{lem:inverselim-top-grp} Let $(G_{n})_{n\in \mathbb{N}}$ be an inverse system of complete topological groups with surjective transition homomorphisms $p_{m,n}\colon G_m\rightarrow G_n$ for every $m\geq n\geq 0$. Then the limit $\mathcal{G}=\varprojlim_{n\in\mathbb{N}}G_{n}$ endowed with the inverse limit topology is a complete topological group and each canonical projection $\widehat{p}_n\colon\mathcal{G}\rightarrow G_n$ is a surjective homomorphism of topological groups. 
\end{lem}
\begin{proof} The fact that $\mathcal{G}$ endowed with the inverse limit topology is a linearly topologized abelian group and the fact that the canonical projections $\widehat{p}_n\colon \mathcal{G}\rightarrow G_n$ are continous homomorphisms are clear. The surjectivity of $\widehat{p}_n$ follows from Mittag-Leffler theorem  \cite[II.3.5 Corollary I]{BourbakiGT}. A countable fundamental system of open subgroups of $\mathcal{G}$ is given for instance by the collection of inverse images of such fundamental systems of each $G_n$ by the homomorphisms $\widehat{p}_n$. Finally, since each $G_n$ is complete, it follows from \cite[II.3.5 Corollary to Proposition 10]{BourbakiGT} that $\mathcal{G}$ is complete. 
\end{proof}

\subsubsection{Families of topological group homomorphism}

\begin{defn} \label{continuous-convergence} Let $G$ and $G'$ be topological groups and let $(f_i)_{i\in I}$ be a family of group homomorphisms from $G$ to $G'$ indexed by a countable set $I$.

a) The family $(f_i)_{i\in I}$ is said to be \emph{uniformly equicontinuous} if for every open subgroup $H'$ of $G'$ there exists an open subgroup  $H$ of $G$ such that $f_i(H)\subset H'$ for all $i\in I$. 

b) The family $(f_i)_{i\in I}$ is said to be \emph{pointwise convergent} to a map $f:G\to G'$ is for every $g\in G$, the family $(f_i(g))_{i\in I}$ converges to $f(g)$.     
\end{defn}

Every element of a uniformly equicontinuous family of group homomorphisms $(f_i)_{i\in I}$ between topological groups $G$ and $G'$ is in particular a homomorphism of topological groups. 

\begin{prop} \label{pro:Uniform_extention}Let $G$ and $G'$ be topological groups with respective separated completions $c:G\to \widehat{G}$ and $c':G'\to \widehat{G'}$ and let $(f_i)_{i\in I}$ be an uniformly equicontinuous family of topological group homomorphisms from $G$ to $G'$. Then the associated family $(\widehat{f_i})_{i\in I}$ of topological groups homomorphisms from $\widehat{G}$ to $\widehat{G'}$ is uniformly equicontinuous. 

Moreover, if the family $(f_i)_{i\in I}$ converges pointwise to a topological group homomorphism $f:G\to G'$ then the family   $(\widehat{f_i})_{i\in I}$ converges pointwise to $\widehat{f}$.
\end{prop}
\begin{proof}
For every $i\in I$, $\widehat{f_i}$ is the unique topological group homomorphism from $\widehat{G}$ to $\widehat{G'}$ such that $\widehat{f_i}\circ c=c'\circ f_i$. 
After replacing $(f_i)_{i\in I}$ by the family $(c'\circ f_i)_{i\in I}$, which is uniformly equicontinuous because $c'$ is uniformly continuous, we can assume without loss of generality from the beginning that $G'$ is complete. Since $G'$
is separated, each $f_i$ uniquely factors through a homomorphism of topological groups $\tilde{f_i}:G/\overline{\{0_G\}}\to G'$, where   $\overline{\{0_G\}}$ denotes the closure of $\{0_G\}$ in $G$ and $G/\overline{\{0_G\}}$ is endowed with the quotient topology.  Since the quotient map $G\to G/\overline{\{0_G\}}$ is open, the family $(\tilde{f}_i)_{i\in I}$ is uniformly equicontinous. So we can assume further without loss of generality from the beginning that $G$ is separated and identify $G$ with its dense image in $\widehat{G}$ by the now injective topological group homomorphism $c:G\to \widehat{G}$. By definition, $\widehat{f_i}$ is then the unique extension of the uniformly continuous topological group homomorphism $f_i:G\to G'$ to the closure $G$ in $\widehat{G}$. Now let $H'$ be an open subgroup of $G'$, let $H$ be any open subgroup of $G$
so that $f_i(H)\subset H'$ for all $i\in I$ and let $\widehat{H}$ be an open subgroup of $\widehat{G}$ such that $G\cap \widehat{H}=H$. Let $\widehat{h}\in \widehat{H}$. Since $H$ is dense in $\widehat{H}$, it follows construction of $\widehat{f_i}$ that for every $i\in I$, there exists $h_i\in H$ such that $\widehat{f_i}(\widehat{h})-f_i(h_i)\in H'$. This implies in turn that $\widehat{f_i}(\widehat{h})=(\widehat{f_i}(\widehat{h})-f_i(h_i))+f_i(h_i)$ belong to $H'$ for every $i\in I$, showing that the family $(\widehat{f_i})_{i\in I}$ is uniformly equicontinuous. The second assertion is straightforward using the uniform continuity of $\hat{f}$ and the $\widehat{f_i}$, $i\in I$.    
\end{proof}

\subsection{Topological rings and modules}

\begin{defn}
    A commutative topological ring $\mathcal{A}$ is a topological abelian group endowed with a ring structure for which the multiplication $\mathcal{A}\times \mathcal{A}\rightarrow \mathcal{A}$ is continuous.  A module $M$ over a topological ring $\mathcal{A}$ is a called a topological $\mathcal{A}$-module if it is a topological abelian group and the scalar multiplication $\mathcal{A}\times M\rightarrow M$ is continuous, where $\mathcal{A}\times M$ is endowed with product topology. 
    \end{defn} 

We denote by $(\mathrm{RTop})$ the category whose object are commutative topological rings $\mathcal{A}$  endowed with a linear topology for
which there exists a fundamental system of neighborhoods of $0_{\mathcal{A}}$
consisting of a countable family $(\mathfrak{a}_{n})_{n\in\mathbb{N}}$
of ideals of $\mathcal{A}$ and whose morphisms are continuous homomorphisms of topological rings. We refer these objects and morphisms simply to as \emph{topological rings} and \emph{homomorphisms of topological rings}, respectively. Similarly, for a topological ring $\mathcal{A}$, we denote by $(\mathrm{MTop}_{\mathcal{A}})$ the category whose objects are topological $\mathcal{A}$-modules $M$ endowed with a linear topology
that admits a fundamental system of neighborhoods of $0_M$ consisting of a
countable family of open submodules $(M_{n})_{n\in\mathbb{N}}$ and whose morphisms are continuous homomorphisms of topological $\mathcal{A}$-modules.

\medskip  

A topological ring or a topological module over a topological ring is called separated (resp. complete) if its separated (resp. complete) as a topological group. We denote by $(\mathrm{CRTop})$ and $(\mathrm{CMTop}_{\mathcal{A}})$ the full subcategories of complete topological rings and complete topological modules over a topological ring $\mathcal{A}$.
Given a topological ring $\mathcal{A}$ (resp. a topological module $M$ over a topological ring $\mathcal{A}$) the separated completion
$\widehat{\mathcal{A}}$ of $\mathcal{A}$ (resp. $\widehat{M}$ of $M$) as a topological group carries the structure of a complete topological ring (resp.of a complete topological $\mathcal{A}$-module) for which the canonical homomorphism $c\colon \mathcal{A}\rightarrow\widehat{\mathcal{A}}$ (resp.
$c\colon M\rightarrow\widehat{M}$) is a homomorphism of topological rings (resp. of topological $\mathcal{A}$-modules). 

\subsubsection{Completed tensor product}\label{completed:tensor}
We recall basic properties of completed tensor products of topological modules, see \cite[III]{BourbakiCA} and  \cite[0.7.7]{GD}.

\begin{prop}[{\cite[III Exercise 28]{BourbakiCA}}] For every pair of topological modules $M$ and $N$ over a topological ring $\mathcal{A}$, the functor $B_{N,M}:(\mathrm{CMTop}_\mathcal{A})\to (\mathrm{Set})$ which associate to a complete topological $\mathcal{A}$-module $E$ the set of $\mathcal{A}$-bilinear homomorphisms $M\times N\to E$ which are continuous with respect to the product topology on $M\times N$ is representable
\end{prop}
\begin{proof}
It is straightforward to verify that $B_{N,M}$ is represented by the pair consisting of the separated completion $\widehat{M\otimes_{\mathcal{A}}N}$
of the tensor product $M\otimes_{\mathcal{A}}N$ with respect to the linear topology
generated by open neighborhoods of $0$ of the form 
$U\otimes N+M\otimes V$, where $U$ and $V$ run respectively through the set of open $\mathcal{A}$-submodules of $M$ and $N$ respectively, and the continuous $\mathcal{A}$-bilinear homomorphism $\tau:M\times N \to \widehat{M\otimes_{\mathcal{A}}N}$ defined as the composition of the canonical homomorphism of topological $\mathcal{A}$-modules  $M\times N\rightarrow M\otimes_{\mathcal{A}}N$ with the separated completion homomorphism $c\colon M\otimes_{\mathcal{A}}N\rightarrow \widehat{M\otimes_{\mathcal{A}}N}$.
\end{proof}

The complete topological $\mathcal{A}$-module representing the functor $B_{N,M}$ is called the \emph{completed tensor product} of $M$ and $N$ is denoted for short by $M\widehat{\otimes}_{\mathcal{A}}N$.
As for the usual tensor
product, its universal property implies the following associativity
result (cf. \cite[II.3.8, Proposition 8]{BourbakiA}):
\begin{lem}
\label{lem:Complete-Tensor-Associative}Let $\mathcal{A}$ be a topological
ring, let $M$ and $\mathcal{B}$ be respectively a topological $\mathcal{A}$-module
and a topological $\mathcal{A}$-algebra and let $N$ and $P$ be
topological $\mathcal{B}$-modules. Then there is a canonical isomorphism
of complete topological $\mathcal{B}$-modules 
\[
(M\widehat{\otimes}_{\mathcal{A}}N)\widehat{\otimes}_{\mathcal{B}}P\cong M\widehat{\otimes}_{\mathcal{A}}(N\widehat{\otimes}_{\mathcal{B}}P)
\]
where $M\widehat{\otimes}_{\mathcal{A}}N$ is viewed as topological $\mathcal{B}$-module
via the $\mathcal{B}$-module structure of $N$.
\end{lem}

In the case where $M=\mathcal{B}_{1}$ and $N=\mathcal{B}_{2}$ are
topological $\mathcal{A}$-algebras, the completed tensor product
$\mathcal{B}_{1}\widehat{\otimes}_{\mathcal{A}}\mathcal{B}_{2}$ is a
complete topological $\mathcal{A}$-algebra and the composition  $\sigma_{1}\colon \mathcal{B}_{1}\rightarrow\mathcal{B}_{1}\widehat{\otimes}_{\mathcal{A}}\mathcal{B}_{2}$
(resp. $\sigma_{2}\colon \mathcal{B}_{2}\rightarrow\mathcal{B}_{1}\widehat{\otimes}_{\mathcal{A}}\mathcal{B}_{2}$)  of  $\mathrm{id}_{\mathcal{B}_{1}}\otimes1\colon \mathcal{B}_{1}\rightarrow\mathcal{B}_{1}\otimes_{\mathcal{A}}\mathcal{B}_{2}$ (resp. $1\otimes\mathrm{id}_{\mathcal{B}_{2}}\colon \mathcal{B}_{2}\rightarrow\mathcal{B}_{1}\otimes_{\mathcal{A}}\mathcal{B}_{2}$) with the separated completion homomorphism $\mathcal{B}_{1}\otimes_{\mathcal{A}}\mathcal{B}_{2}\rightarrow\mathcal{B}_{1}\widehat{\otimes}_{\mathcal{A}}\mathcal{B}_{2}$ is a homomorphism of topological $\mathcal{A}$-algebras. The $\mathcal{A}$-algebra $\mathcal{B}_{1}\widehat{\otimes}_{\mathcal{A}}\mathcal{B}_{2}$ satisfies the following universal property: \emph{For every complete topological $\mathcal{A}$-algebra $\mathcal{C}$
and every pair of homomorphisms of topological $\mathcal{A}$-algebras $f_{i}\colon \mathcal{B}_{i}\rightarrow\mathcal{C}$ there exists a unique homomorphism of topological $\mathcal{A}$-algebras $f\colon \mathcal{B}_{1}\widehat{\otimes}_{\mathcal{A}}\mathcal{B}_{2}\rightarrow\mathcal{C}$ such that $f_{i}=f\circ\sigma_{i}$, $i=1,2$}.

\subsubsection{Separated completed localization}

In what follows by a \emph{multiplicatively closed subset} of a ring $A$, we  mean a subset $S$ of $A$ containing $1$ and stable under multiplication. We now recall basic results on separated completed localizations of topological rings and modules, see \cite[0.7.6]{GD}.  

\begin{prop}\label{def:separated-local}\label{prop:univ-prop-sepcomploc}
For every multiplicatively closed subset $S$ of a topological ring $\mathcal{A}$, the functor $$\mathrm{Hom}_{S^{-1}}(\mathcal{A},-):(\mathrm{CRTop})\to (\mathrm{Sets})$$ 
which associates to every complete topological ring $\mathcal{B}$ the set of homomorphisms of topological rings $\varphi:\mathcal{A}\to \mathcal{B}$ such $\varphi(S)\subset \mathcal{B}^*$ is representable.
\end{prop}
\begin{proof} It is straightforward to verify that $\mathrm{Hom}_{S^{-1}}(\mathcal{A},-)$ is representable by the pair consisting  of the separated completion $\widehat{S^{-1}\mathcal{A}}$ of the usual localization $S^{-1}\mathcal{A}$ endowed with the topology co-induced by the localization homomorphism $j\colon \mathcal{A}\rightarrow S^{-1}\mathcal{A}$ and the homomorphism of topological rings $$\widetilde{j}=c\circ j\colon\mathcal{A}\stackrel{j}{\rightarrow} S^{-1}\mathcal{A} \stackrel{c}{\rightarrow} \widehat{S^{-1}\mathcal{A}}$$
where $c$ is the separated completion homomorphism.
\end{proof}

The complete topological ring $\widehat{S^{-1}\mathcal{A}}$ representing the functor $\mathrm{Hom}_{S^{-1}}(\mathcal{A},-)$ is called the \emph{separated completed localization} of $\mathcal{A}$ with the respect to the multiplicatively closed subset $S$. 

\begin{lem}
\label{lem:Basic-Prop-CompLoc}\label{cor:criterion-zero-loc} Let $\mathcal{A}$ be a topological ring with separated completion $c\colon\mathcal{A}\rightarrow\widehat{\mathcal{A}}$, let $S\subset\mathcal{A}$ be a multiplicatively closed subset and let $\widehat{S}\subset\widehat{\mathcal{A}}$ be the closure of $c(S)$ in $\widehat{\mathcal{A}}$. Then there exists a canonical isomorphism  $\widehat{S^{-1}\mathcal{A}}\cong \widehat{\widehat{S}^{-1}\widehat{\mathcal{A}}}$ of complete topological rings. 
In particular, $\widehat{S^{-1}\mathcal{A}}$
is the zero ring if and only if  $0_{\widehat{\mathcal{A}}}$ belongs to the closure $\widehat{S}$
of $c(S)$ in $\widehat{\mathcal{A}}$.

\end{lem}

\begin{proof}
Let $(\mathfrak{a}_n)_{n\in \mathbb{N}}$ be a fundamental system of open ideal in $\mathcal{A}$, let $\mathrm{p}_n\colon\mathcal{A}\rightarrow A_n=\mathcal{A}/\mathfrak{a}_n$, $n\in \mathbb{N}$, be the quotient homomorphisms and let $S_n=\mathrm{p}_n(S)\subset A_n$.
Note that for each $n$ there is a canonical isomorphism
$
S_n^{-1}A_n \;\cong\; (S^{-1}\mathcal A)/(S^{-1}\mathfrak a_n),
$
so that the rings $S_n^{-1}A_n$ are precisely the quotients of $S^{-1}\mathcal A$ by a fundamental
system of open ideals.
Then $\widehat{S}=\varprojlim_{n\in \mathbb{N}} S_n\subset \varprojlim_{n\in \mathbb{N}} A_n=\widehat{\mathcal{A}}$ so that, by definition,  $$\widehat{{S}^{-1}\mathcal{A}}\cong \varprojlim_{n\in \mathbb{N}}S_n^{-1} A_n\cong \widehat{\widehat{S}^{-1}\widehat{\mathcal{A}}}.$$
For the second assertion, we can reduce without loss of generality to the case where $\mathcal{A}$ is complete and $S$ is closed in $\mathcal{A}$. Now if $0_{\mathcal{A}}\in S$ then $S^{-1}\mathcal{A}$ is the zero ring, and so $\widehat{S^{-1}\mathcal{A}}$ is the zero ring as well. Conversely, if $\widehat{S^{-1}\mathcal{A}}=\varprojlim_{n\in \mathbb{N}} S_n^{-1}A_n$ is the zero ring, then $S_n^{-1}A_n$ is the zero ring for every $n\in \mathbb{N}$, which implies that $0_{\mathcal{A}}\in S_n$ for every $n\in \mathbb{N}$. It follows that $0$ belongs  $\varprojlim_{n\in \mathbb{N}} S_n=S$ as $S$ is closed. 
\end{proof}

\begin{lem} \label{lem:inject-loc}
\label{lem:Injectivity-Loc}Let $i\colon\mathcal{A}\rightarrow\mathcal{B}$ be an injective \emph{closed} homomorphism of complete topological rings
and let $S$ be a multiplicatively closed subset of $\mathcal{A}$. Then the naturally induced map $\widehat{S^{-1}i}\colon \widehat{S^{-1}A}\rightarrow\widehat{i(S)^{-1}\mathcal{B}}$ is an injective homomorphism of topological rings.
\end{lem}

\begin{proof}
Since $\mathcal{A}$ (resp. $\mathcal{B}$) is complete, the kernel of the separated completed localization homomorphism $\mathcal{A}\rightarrow \widehat{S^{-1}\mathcal{A}}$ (resp. $\mathcal{B}\rightarrow \widehat{i(S)^{-1}\mathcal{B}}$) consists of elements of $\mathcal{A}$ (resp. $\mathcal{B}$) which are anihilated by the multiplication by an element of the closure of $S$ in $\mathcal{A}$ (resp. of the closure of $i(S)$ in $\mathcal{B}$). On the other hand, since $i$ is a closed homomorphism of topological rings, $i(\mathcal{A})$ is complete subspace of $\mathcal{B}$, hence a closed subspace, so that the closures of $i(S)$ in $i(\mathcal{A})$ and $\mathcal{B}$ coincide. 
\end{proof}

\begin{example} 
Let $\mathcal{B}=\mathbb{C}[u]$
endowed with the $u$-adic topology, with fundamental system of neighbourhoods
of $0_{\mathcal{B}}$ given by the ideals $\mathfrak{b}_{n}=u^{n}\mathbb{C}[u]$,
$n\geq0$, and let $S=\{u^{m}\}_{m\geq0}$. We have then $\widehat{\mathcal{B}}\cong\mathbb{C}[[u]]$
endowed with the $u$-adic topology and $\widehat{S}=\{u^{m}\}_{m\geq0}\cup\{0\}$. It follows that $\widehat{S}^{-1}\widehat{\mathcal{B}}$ whence $\widehat{\widehat{S}^{-1}\widehat{\mathcal{B}}}$ is the zero ring. 
On the other hand, we have $S^{-1}\mathcal{B}=\mathbb{C}[u^{\pm1}]$
and the images of the ideal $\mathfrak{b}_{n}$ 
by the localization homomorphism
are all equal to the unit ideals in $\mathbb{C}[u^{\pm1}]$
The induced topology on $S^{-1}\mathcal{B}$ 
is thus the trivial ones, which implies that the separated completion $\widehat{S^{-1}\mathcal{B}}$ is the zero ring.

 Note that the conclusion of Lemma  \ref{lem:inject-loc} does not hold if  $i$ is not a closed homomorphism: the inclusion $i:\mathcal{A}=\mathbb{C}[u]\rightarrow \widehat{\mathcal{B}}=\mathbb{C}[[u]]$, where $\mathbb{C}[u]$ and $\mathbb{C}[[u]]$ are endowed respectively with the discrete topology and the $u$-adic topology,  is continuous but not closed, and for $S=\{u^{m}\}_{m\geq0}$, the separated completed localizations $\widehat{ S^{-1}\mathcal{A}}$ and $\widehat{\widehat{S}^{-1}\widehat{\mathcal{B}}}$ are respectively isomorphic to $\mathbb{C}[u,u^{-1}]$ endowed with the discrete topology and to the zero ring, so that $\widehat{S^{-1}i}$ is not injective in this case.  
\end{example}

%
%

\subsection{Restricted power series}\label{ref:Restricted power series}

\begin{defn}[{\cite[III.4.2]{BourbakiCA}} and {\cite[0.7.5]{GD}}] \label{def:res-ps}
Let $\mathcal{A}$ be a topological ring with separated completion  $c\colon \mathcal{A}\rightarrow\widehat{\mathcal{A}}$.  The ring of \emph{restricted power series} with coefficients in $\widehat{\mathcal{A}}$ and indeterminates $T_{1},\ldots,T_{r}$ is the separated completion $\widehat{\mathcal{A}}\{T_{1},\ldots,T_{r}\}$ of the polynomial ring $\mathcal{A}[T_1,\ldots, T_r]$ endowed with the topology generated by the ideals $\mathfrak{a}[T_1,\ldots , T_r]$, where $\mathfrak{a}$ runs throught the set of open ideals of $\mathcal{A}$.

The elements in the image of the natural homomorphism of topological rings $i_0\colon\widehat{\mathcal{A}}\rightarrow \widehat{\mathcal{A}}\{T_1, \ldots, T_r\}$ deduced from the inclusion $\mathcal{A}\hookrightarrow \mathcal{A}[T_1,\ldots , T_r]$ as the subring of constant polynomials are called \emph{constant restricted power series}.  
\end{defn}

Letting $(\mathfrak{a}_{n})_{n\in\mathbb{N}}$ be a fundamental system of open ideals of $\mathcal{A}$, it follows from the definition that 
\[
\widehat{\mathcal{A}}\{T_{1},\ldots,T_{r}\}\cong\varprojlim_{n\in\mathbb{N}}(\mathcal{A}/\mathfrak{a}_{n})[T_{1},\ldots,T_{r}]\cong \mathcal{A}\widehat{\otimes}_{\mathbb{Z}}\mathbb{Z}[T_1,\ldots, T_n] \cong \widehat{\mathcal{A}}\widehat{\otimes}_{\mathbb{Z}}\mathbb{Z}[T_1,\ldots, T_n].
\]

Identifying a polynomial in $\widehat{\mathcal{A}}[T_1,\ldots, T_r]$ with the family $(a_{I})_{I\in\mathbb{N}^{r}}$ of its coefficients, the elements of $\widehat{\mathcal{A}}\{T_1, \ldots, T_r\}$ are represented in turn by families $(a_{I})_{I\in\mathbb{N}^{r}}$ of elements of $\widehat{\mathcal{A}}$ which converge to $0$. 
This allows to further identify  $\widehat{\mathcal{A}}\{T_1, \ldots, T_r\}$ with the $\widehat{\mathcal{A}}$-subalgebra of the algebra of formal power series $\widehat{\mathcal{A}}[[T_1,\ldots, T_n]]\cong \prod_{I\in \mathbb{N}^n}\widehat{\mathcal{A}} $ with coefficients in $\widehat{\mathcal{A}} $
consisting of formal power series 
\[
\sum_{I=(i_{1},\ldots,i_{r})\in\mathbb{N}^{r}}a_{I}T_{1}^{i_{1}}\cdots T_{r}^{i_{r}}
\]
such that the family $(a_{I})_{I\in\mathbb{N}^{r}}$ converges to
$0$, see \cite[III.4.2]{BourbakiCA}.

\begin{lem} \label{lem:cc-converge-to-restricted-map}
Let $\mathcal{A}$ be a topological ring, let $\mathcal{B}$ be a separated topological ring with separated completion $c\colon\mathcal{B}\rightarrow \widehat{\mathcal{B}}$, let $h_n\colon \mathcal{A}\rightarrow \mathcal{B}$, $n\geq 0$, be a sequence of homomorphisms of groups and let $$\sigma :\mathcal{A}\to \widehat{\mathcal{B}}[[T]], \;a\mapsto \sum_{n\geq 0} c(h_n(a))T^n$$
be the associated homomorphisms of $\mathcal{A}$-algebras. Then the following are equivalent:

(a) The homomorphism $\sigma$ factors through a homomorphism of topogical $\mathcal{A}$-algebras $s:\mathcal{A}\to \widehat{\mathcal{B}}\{T\}$,

(b) The family $(h_n)_{n\geq 0}$ is uniformly equicontinuous and pointwise convergent to the zero homomorphism. 
\end{lem}

\begin{proof}
Since $c$ is uniformly continuous, the family $(h_n)_{n\geq 0}$ is uniformly equicontinous and pointwise convergent to zero if and only if so does the family $(c\circ h_n)_{n\geq 0}$. So we can assume without loss of generality that $\mathcal{B}$ is complete. Under the identification above of $\mathcal{B}\{T\}$ as  sub-$\mathcal{A}$-algebra of $\mathcal{B}[[T]]$, the homomorphism $\sigma$ factorizes through a homomorphism $\mathcal{A}\to \mathcal{B}\{T\}$ if and only if the family $(h_n)_{n\geq 0}$ is pointwise convergent to the zero homomorphism. When this holds, it follows from the definition of the topology on $\mathcal{B}\{T\}\cong \mathcal{B}\widehat{\otimes}_{\mathbb{Z}} \mathbb{Z}[T]$ that 
the induced homomorphism $s:\mathcal{A}\to \mathcal{B}\{T\}$
is continuous if and only if for every open ideal $\mathfrak{b}$ of $\mathcal{B}$, there exists an open ideal $\mathfrak{a}$ of $\mathcal{A}$ such that $h_n(\mathfrak{a})\subset \mathfrak{b}$, that is, if and only if the family $(h_n)_{n\geq 0}$ is uniformly equicontinuous. 
\end{proof}

We now collect additional basic results on restricted power series rings.

\begin{prop}[{\cite[III.4.2 Proposition 4]{BourbakiCA}}]\label{prop:univ-prop-restricted} The ring $\widehat{\mathcal{A}}\{T_{1},\ldots,T_{r}\}$
satisfies the following universal property: for every continuous
ring homomorphism $f\colon \mathcal{A}\rightarrow\mathcal{B}$ to a complete
topological ring $\mathcal{B}$ and every choice of $r$ elements
$b_{1},\ldots,b_{r}$ of $\mathcal{B}$, there exists a unique continuous ring 
homomorphism $\overline{f}\colon \widehat{\mathcal{A}}\{T_{1},\ldots,T_{r}\}\rightarrow\mathcal{B}$
such that $\overline{f}|_{c(\mathcal{A})}=\widehat{f}$ and such that $\overline{f}(T_{i})=b_{i}$
for every $i=1,\ldots,r$.
\end{prop}

\begin{lem}\label{lem:projlim-restricted}
For every complete topological ring $\mathcal{A}$ and every set of
variables $T_{1},\ldots,T_{s},T_{s+1},\ldots,T_{r}$, there exist
canonical isomorphisms of complete topological $\mathcal{A}$-algebras
\[
\mathcal{A}\{T_{1},\ldots,T_{s},T_{s+1},\ldots T_{r}\}\cong\mathcal{A}\{T_{1},\ldots,T_{r}\}\widehat{\otimes}_{\mathcal{A}}\mathcal{A}\{T_{s+1},\ldots,T_{r}\}\cong\mathcal{A}\{T_{1},\ldots,T_{s}\}\{T_{s+1},\ldots T_{r}\}.
\]
\end{lem}
\begin{proof} Straightforward from Lemma \ref{lem:Complete-Tensor-Associative}.
\end{proof}
\begin{lem}\label{lem:proj-lim-restricted} Let $\mathcal{B}$ be the limit of a countable inverse system $(\mathcal{B}_n)_{n\in \mathbb{N}}$ of complete topological rings with surjective continuous transition homomorphisms $p_{m,n}\colon\mathcal{B}_m\rightarrow \mathcal{B}_n$ for every $m\geq n\geq 0$ and let $T_{1},\ldots,T_{r}$ be indeterminates. Then the canonical homomorphism of complete topological rings 
$$\mathcal{B}\{T_1,\ldots, T_r\}\cong (\varprojlim_{n\in \mathbb{N}}\mathcal{B}_n)\widehat{\otimes}_{\mathbb{Z}}\mathbb{Z}[T_1,\ldots , T_r] \rightarrow \varprojlim_{n\in \mathbb{N}} (\mathcal{B}_n \widehat{\otimes}_{\mathbb{Z}} \mathbb{Z}[T_1,\ldots T_r])\cong \varprojlim_{n\in \mathbb{N}} (\mathcal{B}_n\{T_1,\ldots , T_r\})$$
is an isomorphism.
\end{lem}
\begin{proof} Let $p_n\colon\mathcal{B}\rightarrow \mathcal{B}_n$, $n\in \mathbb{N}$ be the canonical projection homomorphisms.  By definition, elements of $\mathcal{B}\{T_1,\ldots ,T_r\}$ are represented by families $(b_{I})_{I\in\mathbb{N}^{r}}$ of elements of $\mathcal{B}$ which converge to $0$ in $\mathcal{B}$. Since the projection homomorphisms $p_n$ are surjective, it follows from the definition of the topology on $\mathcal{B}$ that these families are in one-to-one correspondence with collections of families $(b_{n,I})_{I\in\mathbb{N}^{r}}$ of elements of $\mathcal{B}_n$, $n\in \mathbb{N}$, such that $(b_{n,I})_{I\in\mathbb{N}^{r}}$ converges to $0$ in $\mathcal{B}_n$ for every $n\in \mathbb{N}$ and such that $b_{n,I}=p_{m,n}(b_{m,I})$ for every $m\geq n\geq 0$ and every $I\in \mathbb{N}^r$. 
\end{proof}

\begin{lem}
\label{lem:Spe-Comp-Loc-Restricted-degree-0}Let $\mathcal{A}$ be
topological ring with separated completion $c\colon\mathcal{A}\rightarrow\widehat{\mathcal{A}}$
and let $S\subset\mathcal{A}$ be a multiplicatively closed subset. Let $T_{1},\ldots,T_{r}$ be a set of indeterminates and
let $\widehat{S}_{0}\subset\mathcal{\widehat{\mathcal{A}}}\{T_{1},\ldots,T_{r}\}$
be the image of $S$ by the composition of $c$ with the inclusion
$i_{0}\colon\widehat{\mathcal{A}}\hookrightarrow\mathcal{\widehat{\mathcal{A}}}\{T_{1},\ldots,T_{r}\}$. Then there exists
a canonical isomorphism of complete topological $\widehat{S^{-1}\mathcal{A}}$-algebras
\[
\widehat{S^{-1}\mathcal{A}}\{T_{1},\ldots,T_{r}\}\cong\widehat{(\widehat{S}_{0}^{-1}(\mathcal{\widehat{A}}\{T_{1},\ldots,T_{r}\}))},
\]
where $\widehat{(\widehat{S}_{0}^{-1}(\mathcal{\widehat{A}}\{T_{1},\ldots,T_{r}\}))}$
is viewed as an $\widehat{S^{-1}\mathcal{A}}$-algebra via the unique homomorphism of topological rings deduced from the homomorphism of topological rings 
\[
\mathcal{A}\stackrel{i_{0}\circ c}{\rightarrow}\mathcal{\widehat{A}}\{T_{1},\ldots,T_{r}\}\rightarrow\widehat{(\widehat{S}_{0}^{-1}(\mathcal{\widehat{A}}\{T_{1},\ldots,T_{r}\}))}
\]
by the universal property of separated completed localization.
\end{lem}

\begin{proof}
By definition of the separated completed localization and the of the restricted power series rings, these topological rings are both isomorphic, as topological $\widehat{S^{-1}\mathcal{A}}$-algebras, to the separated completion of the ring $S^{-1}\mathcal{A}[T_1,\ldots ,T_r]=S^{-1}\mathcal{A}\otimes_{\mathbb{Z}} \mathbb{Z}[T_1,\ldots , T_r]$ with respect to the topology generated by the open ideals $S^{-1}\mathfrak{a}[T_1,\ldots ,T_r]$, where $\mathfrak{a}$ ranges through the set of open ideals of $\mathcal{A}$. 
\end{proof}

\end{document}